\documentclass[letterpaper,11pt]{article}
\usepackage[utf8x]{inputenc}
\usepackage[title]{appendix}
\usepackage[noTeX]{mmap}
\usepackage[left=1.3in, right=1.3in, bottom = 1.4in, top = 1.4in]{geometry}
\usepackage{hyperref} 
\hypersetup{colorlinks}
\usepackage{amsmath,amsfonts, amsthm, amssymb}
\usepackage{color}
\usepackage{enumerate}
\usepackage{hypbmsec}
\usepackage{bbm}
\usepackage{dsfont}
\usepackage{mathtools}
\usepackage{graphicx}
\usepackage{caption}
\usepackage{subcaption}
\usepackage[section]{placeins}
\usepackage{tikz}
\usetikzlibrary{calc}
\usetikzlibrary{decorations.pathreplacing}
\usetikzlibrary{arrows}
\usepackage{pgffor}
\usepackage{longtable}
\usepackage{comment}
\usepackage{multirow}
\usepackage{rotating}
\usepackage{bm}
\usepackage{mdframed}

\usepackage{fancyhdr}
\usepackage{tocloft}

\usepackage{mathtools,tabstackengine}
\TABstackMath
\setstacktabbedgap{10pt}
\setstackgap{L}{1.5\baselineskip}

\theoremstyle{plain}
\newtheorem{theorem}{Theorem}[section]
\newtheorem{proposition}[theorem]{Proposition}
\newtheorem{corollary}[theorem]{Corollary}
\newtheorem{lemma}[theorem]{Lemma}
\newtheorem{definition}[theorem]{Definition}
\newtheorem{example}[theorem]{Example}

\theoremstyle{definition}
\newtheorem{remark}[theorem]{Remark}

\newcommand{\AS}{\xrightarrow{a.s.}}

\newcommand{\pref}[1]{\hyperref[#1]{[p. \pageref{#1}]}}

\newcommand{\old}[1]{}

\newcommand{\defeq}{\vcentcolon=}

\newcommand{\distto}{%
	\mathrel{\vbox{\offinterlineskip\ialign{%
				\hfil##\hfil\cr
				$\scriptscriptstyle\mathrm{d}$\cr
				\noalign{\kern-.05ex}
				$\to$\cr
}}}}
\newcommand{\stabto}{%
	\mathrel{\vbox{\offinterlineskip\ialign{%
				\hfil##\hfil\cr
				$\scriptscriptstyle\mathrm{st}$\cr
				\noalign{\kern-.05ex}
				$\to$\cr
}}}}
\newcommand{\findimto}{%
	\mathrel{\vbox{\offinterlineskip\ialign{%
				\hfil##\hfil\cr
				$\scriptscriptstyle\mathrm{f.d.}$\cr
				\noalign{\kern-.05ex}
				$\to$\cr
}}}}

\newcommand{\Probto}{%
	\mathrel{\vbox{\offinterlineskip\ialign{%
				\hfil##\hfil\cr
				$\scriptscriptstyle\Prob$\cr
				\noalign{\kern-.05ex}
				$\to$\cr
}}}}
\newcommand{\TVto}{%
	\mathrel{\vbox{\offinterlineskip\ialign{%
				\hfil##\hfil\cr
				$\scriptscriptstyle\mathrm{TV}$\cr
				\noalign{\kern-.05ex}
				$\to$\cr
}}}}

\newcommand{\ind}[1]{\mathds{1}_{\{#1\}}}

\newcommand{\N}{\mathds{N}}
\newcommand{\Z}{\mathds{Z}}

\newcommand{\R}{\mathds{R}}

\newcommand{\C}{\mathds{C}}
\newcommand{\T}{\mathds{T}}

\newcommand{\U}{\mathcal{U}_{\infty}}

\newcommand{\imag}{\mathrm{i}}

\newcommand{\Prob}{\mathds{P}}
\DeclareMathOperator{\Var}{\mathrm{Var}}
\DeclareMathOperator{\Cov}{\mathrm{Cov}}

\newcommand{\E}{\mathds{E}}

\renewcommand{\L}{\mathcal{L}}
\newcommand{\F}{\mathcal{F}}

\newcommand{\cG}{\mathcal{G}}

\newcommand{\eqdist}{%
	\mathrel{\vbox{\offinterlineskip\ialign{%
				\hfil##\hfil\cr
				$\scriptscriptstyle\mathrm{law}$\cr
				\noalign{\kern.2ex}
				$=$\cr
}}}}

\newcommand{\supp}{\mathrm{supp}}

\newcommand{\dt}{\mathrm{d} \mathit{t}}

\newcommand{\type}{\mathrm t}

\title{Limit theorems for discrete multitype branching processes counted with a characteristic}

\author{Konrad Kolesko and  Ecaterina Sava-Huss}

\date{\today}

\begin{document}
\maketitle

\begin{abstract}
For a discrete time multitype supercritical Galton-Watson process $(Z_n)_{n\in \N}$ 
and corresponding genealogical tree ${\T}$, we associate a new discrete time process $(Z_n^{\Phi})_{n\in\N}$ such that, for each $n\in \N$, the contribution of each individual $u\in\T$ to $Z_n^{\Phi}$ is determined  by a (random) characteristic $\Phi$ evaluated at the age of $u$ at time $n$. In other words, $Z_n^{\Phi}$ is obtained by summing over all $u\in \T$ the corresponding contributions $\Phi_u$, where $(\Phi_u)_{u\in \T}$ are i.i.d. copies of $\Phi$. Such processes are known in the literature under the name of {\it Crump-Mode-Jagers (CMJ)  processes counted with characteristic $\Phi$}. We derive a LLN and a CLT for the process $(Z_n^{\Phi})_{n\in\N}$ in the discrete time setting, and in particular, we show a dichotomy in its limit behavior. By applying our main result, we also obtain a generalization of the results in Kesten-Stigum \cite{kesten-stigum-add}. 
\end{abstract}


\section{Introduction}
For a fixed natural number $J$, the general Crump-Mode-Jagers (or shortly CMJ) branching process with $J$ types is a classical model
for an evolving population, in which for a collection $\Xi\defeq\{\xi^{(i,j)}\}_{i,j\le J}$ of point processes on $(0,\infty)$, the evolution  goes as follow.
We start with the initial ancestor $\varnothing$ of fixed type $i_0\in\{1,\dots,J\}$ at time $0$, and  for any $j\in\{1,\dots,J\}$, $\varnothing$ produces offspring of type $j$ at the times of the point process $\xi^{(i_0,j)}$.
Every other individual in the process reproduces according to an independent copy of $\Xi$ shifted by the individual's time of birth.
Thus  individuals that are ever born form a multitype Galton-Watson tree $\T$. For any $u\in\T$, by $\type(u)$ we denote the type of $u$, i.e.~a number from the set $\{1,\dots,J\}$ and by $S(u)$ we denote the time when the individual $u$ is born. In addition, for any $u\in\T$ we assign a function $\Phi_u$ which is an independent copy of some given random function
$\Phi:(0,\infty)\times \{1,\dots,J\}\to\R$.
The {\it multitype Crump-Mode-Jagers (CMJ) process counted with characteristic $\Phi$} is the process $(Z^\Phi_t)_{t\geq 0}$ defined as
$$Z^\Phi_t\defeq\sum_{u\in\T}\Phi_u(t-S(u),\type(u)),$$
that is, the contribution of a given individual $u\in\T$ to $Z^{\Phi}_t$ involves  the assigned function $\Phi_u$, the age of $u$ at time $t$ and the type of $u$.
This model extends a variety of other models such as multitype Galton-Watson processes,
age-dependent branching processes, Bellman-Harris processes and Sevast'yanov processes.
When counted with a random characteristic, the general branching process offers a lot of flexibility in modeling
and allows considerations of, for instance, the number of individuals in the population
in some random phase of life or having some random age-dependent property.
We refer the reader to \cite{Jagers:1975} for a textbook introduction. 

First order asymptotics of $(Z^\Phi_t)_{t\geq 0}$ are by now quite well understood.
The strong law of large numbers in the single type case $J=1$ has been proven by Nerman \cite{Nerman:1981}, while the lattice version of Nerman's law of large numbers was proved by Gatzouras \cite{Gatzouras:2000}. The general type space (even infinite) has been studied by Jagers \cite{jagers} who proved the weak law of large numbers, while the corresponding strong law of large numbers was provided in Olofsson \cite{olofsson2009}; see \cite{iksanov-meiners} for an alternative proof. These results can be summarized by saying that under suitable conditions on $\Xi$ and on  the (random) characteristic $\Phi$ there exists a parameter $\alpha>0$, called \emph{the Maltusian parameter}, such that $e^{-\alpha t}Z^{\Phi}_t$ converges as $t\to\infty$  almost surely or in probability to some random variable $X^\Phi$. It is then natural to ask about lower order terms in the asymptotic expansion $Z^\Phi_t=e^{\alpha t}X^\Phi+o(e^{\alpha t})$.

The fluctuations of single-type CMJ processes counted with characteristic $\Phi$ have been meanwhile investigated in a sequence of works. One particular example is the  Kolmogorov's conservative fragmentation model that may be translated into the language of general branching processes. For such processes,  Janson and Neininger \cite{Janson+Neininger:2008} proved a corresponding limit theorem.
Janson \cite{Janson:2018} studied asymptotic fluctuations of single-type supercritical general branching processes
in the lattice case  for some particular choice of the characteristic in which the limit in Nerman's theorem from \cite{Nerman:1981} vanishes. Recently, central limit theorems for such general single type branching processes have been established in  \cite{iksanov-kolesko-meiners2021}. 

Much less is known in the multitype case, i.e.~when $J>1$. 
There are partial results indicating the intricate nature of the limit theorems  that can occur.
For multitype continuous-time Markov branching processes with finite type space, in which
individuals give birth only at the time of their death,
Athreya \cite{athreya-cont1,athreya-cont2} proved a central limit theorem
and Janson \cite{janson-fclt-urns} proved a corresponding functional central limit theorem.
Asmussen and Hering \cite[Section VIII.3]{Asmussen+Hering:1983} provide
results for the asymptotic fluctuations of multitype Markov branching processes
with rather general type space.
Let us remark that although most of the   above mentioned results have different proofs, they are somehow similar in nature, that is, some kind of dichotomy/trichotomy appears in the limit behavior. This may indicate the existence of a limit theorem for general CMJ processes $(Z^\Phi_t)_{t\geq 0}$ that generalizes the previous results.

\textbf{Our contribution.} In the underlying work we take a step forward in understanding higher order asymptotics for multitype  CMJ processes counted with characteristics $\Phi$. The main restriction imposed in the current paper is that for all $1\le i,j\le J$,  
\begin{align}
	\label{eq:supported at 1}
	\operatorname{\supp} \xi^{(i,j)}\subseteq\{1\},
\end{align}
that is, each particle can give birth to new particles only at age one. More precisely, we consider a supercritical multitype Galton-Watson branching process $(Z_n)_{n\in\N}$ with $J$ types $Z_n=(Z_n^1,...,Z_n^J)$, where for $j\in\{1,\ldots,J\}$, $Z_n^j$ stands for the number of particles of type $j$ born in the $n$-th generation. We denote by $\T$ the corresponding genealogical Galton-Watson tree. For each individual $u\in\T$ ever born we associate a weight function $\Phi_u$, which is an independent (of everything) copy  of some given random function  $\Phi:\Z\times \{1,\dots,J\}\to \C^{d}$
that we call {\it characteristic of dimension $d\in\N$}.  
This work focuses on asymptotic fluctuations  of the discrete time stochastic process $(Z_n^{\Phi})_{n\in \N}$ defined as
$$Z_n^{\Phi}=\sum_{u\in \T}\Phi_u(n-|u|,{\type(u)}),$$
where  $|u|$  stands for the generation  of the individual $u\in\T$. 
We keep the established terminology and call $(Z_n^{\Phi})_{n\in \N}$ \textit{branching process counted with characteristic $\Phi$}. 
Our setup and assumptions regarding the process $(Z_n^{\Phi})_{n\in \N}$ are described in detail in Section \ref{sec:branch-char}.
Assuming positive regularity and existence of second
moments for the multitype branching process $(Z_n)_{n\in \N}$, and mild conditions on the characteristic $\Phi$, we establish in Theorem~\ref{thm:main} a dichotomy in the asymptotic behavior of $(Z_n^{\Phi})_{n\in\N}$.
The dichotomy is based on the existence of eigenvalue $\lambda$ of the mean offspring matrix $A\defeq (\E\xi^{(i,j)}((0,\infty)))_{1\le i,j\le J}$ such that $|\lambda|=\sqrt\rho$.

The motivation for this work is to demonstrate how the methods from \cite{iksanov-kolesko-meiners2021} can be applied in order to study particular multitype CMJ processes for which \eqref{eq:supported at 1} hold, and 
one possible application of our results is in the context of urn models in  \cite{kolesko-sava-huss-urn}.  We also obtain a generalization of the results from \cite{kesten-stigum-add}, by applying the main result Theorem~\ref{thm:main} to the particular choice of the characteristic $\Phi(k)\defeq \mathbf{1}\cdot\ind{k=0}$.

This work can  be seen as a first step forward in  answering 
Question 1 from \cite{iksanov-kolesko-meiners2021}.  The choice of our model allows us to avoid many  issues that arise in the general case, that is without condition \eqref{eq:supported at 1}, both in lattice and non-lattice case (i.e. when all $\xi^{(i,j)}$ are supported on some fixed lattice or not).
For example, under condition  \eqref{eq:supported at 1} we can easily establish the asymptotic expansion of the mean $\E Z_n^\Phi$ in terms of the spectral decomposition of the mean offspring matrix $A$ and this expansion is then required to conclude the limit theorem for $Z_n^{\Phi}$. On the other hand, the fact that the branching processes studied in this paper are multidimensional causes additional complications not present in \cite{iksanov-kolesko-meiners2021}, e.g. the presence of a nilpotent part in the Jordan decomposition of $A$ gives rise to a polynomial  correction in the limit theorem.

Let us emphasize that the role of the eigenvalues of $A$ in the one-type case \cite{iksanov-kolesko-meiners2021} are played by, after taking the exponent,   the solutions of the equation
$$\int_0^\infty e^{-zt}\E\xi(\dt)=1.$$
Thus, a limit theorem for a general   multitype CMJ processes should cover both described above cases and it is natural to expect that in this case the corresponding exponents will be given by the solutions of the following matrix equation:
$$\int_0^{\infty}e^{-zt}\E\xi^{(i,j)}(dt)=\ind{i=j}\quad\text{for all }i,j\in\{1,\dots,J\}.$$
However, at the moment, even this exact statement is unclear and  the limit theorem for a general multiple CMJ processes (both lattice and non-lattice case) will be the subject of future research.

Finally, the investigation of such processes has yet another not obvious motivation in the context of rotor walks with random initial configuration on trees, see \cite{rotor-GWtrees,range-rotor2,range-rotor1}. In these papers the authors  investigate law of large numbers for a class of processes called \textit{rotor walks},  processes that interpolate between random walks and deterministic walks; see also \cite{huss-levine-sava} for a functional limit theorem of the interpolated processes. The range $R_n$ (number of distinct sites visited by the walk up to time $n$) is strongly related to the size of random trees that are constructed recursively from Galton-Watson trees.

\textbf{Structure of the paper.} In Section \ref{sec:prelim} we fix the notation, introduce Ulam-Harris trees, multitype Galton-Watson processes,  and we describe the spectral decomposition of the mean value matrix of the multitype Galton-Watson processes. Then, in Section \ref{sec:branch-char} we introduce \emph{branching processes counted with a characteristic} and the natural assumptions for the characteristic, and we state our main result Theorem \ref{thm:main}. Afterwards, in order to prove the main result, we proceed by proving a series of auxiliary results concerning branching processes with deterministic and with random centered characteristics in Section \ref{sec:aux-results}. The last Section \ref{sec:proof-main-result} is dedicated to the proof of Theorem \ref{thm:main}. Finally, as an application of our main results, we show how to extend the results in Kesten-Stigum \cite{kesten-stigum-add} by a specific choice of the characteristic.

\section{Preliminaries}\label{sec:prelim}

\textbf{Conventions and notations.} 
The number $J\in \N$ will be reserved for the number of types of the multitype Galton-Watson process, and we will often write $[J]$ for the set $\{1,\ldots,J\}$.
All our vectors will be regarded as column vectors, and for $\mathrm{u}\in \R^J$  the notation $\mathrm{u}=(\mathrm{u}_1,\ldots,\mathrm{u}_J)$ is purely for convenience. When we deal with row vectors, it will be either mentioned explicitly, or it will be clear from the context. 
For $j\in [J]$, we denote by  $\mathrm{e}_j=(0,\ldots,1,\ldots,0)$ the $j$-th standard basis vector in $\R^J$.

For a $n\times m$ matrix $A=(a_{ij})_{i,j}$, with $m,n\in\N$, the Hilbert–Schmidt norm of $A$, called also Frobenius norm, is defined as
\vspace{-0.25cm}
 $$\|A\|_{HS}\defeq \Big(\sum_{i=1}^n\sum_{j=1}	^m|a_{ij}|^2\Big)^{1/2}.$$
Since for any vector the Hilbert-Schmidt norm coincides with the Euclidean norm,  for the rest of the paper we shall only write $\|\cdot\|$ instead of $\|\cdot\|_{HS}$.  

We fix a probability space $(\Omega,\F,\Prob)$ on which all the random variables and processes we consider are defined.
For a random matrix $L=(L_{ij})_{i,j}$ (of any size) we define a deterministic matrix $\Var[L]$ of the same size as $L$ by taking variances of each entry in $L$:
 $$\Var[L]:=(\Var[L]_{ij})_{i,j}=(\Var[L_{ij}])_{i,j}.$$
 


%

\subsection{Multitype branching processes}

\textbf{The Ulam-Harris tree.}
It is convenient when investigating limit theorems for Galton-Watson processes and trees, to look at such trees as subtrees of the infinite Ulam-Harris  tree $\U$ that we define now. The vertex set of $\U$ is  $V_{\infty}\defeq \bigcup_{n \in \N_0} \N^n$, the set of all finite strings or words $v_1\cdots v_n$ of positive integers over $n$ letters, including the empty word $\varnothing$ which we take to be the root, and with an edge joining $v_1\cdots v_n$ and $v_1\cdots v_{n+1}$ for any $n\in \N_0$ and any $v_1, \cdots, v_{n+1}\in \N$. Thus every vertex $v=v_1\cdots v_n$ has outdegree $\infty$, and the children of $v$ are the words $v1,v2,\ldots$ and we let them have this order so that $\U$ becomes an infinite ordered rooted tree. We will identify $\U$ with its vertex set $V_{\infty}$, when no confusion arises and for vertices $v=v_1\cdots v_n$ we also write $v=(v_1,\ldots,v_n)$, and if $u=(u_1,\ldots,u_m)$ we write $uv$ for the concatenation of the words $u$ and $v$, that is  $uv=(u_1,\ldots,u_m,v_1,\ldots,v_n)$. The parent of $v_1\cdots v_n$ is $v_1\cdots v_{n-1}$.
Further, if $k \leq m$, we set $u|_k \defeq u_1 \ldots u_k$ for the vertex $u$ truncated at height $k$.
Finally, for $u \in \U$, we use the notation $|u|=n$ for $u \in\N^n$, i.e. $u$ is a word of length (or height) $n$, that is, at distance $n$ from the root $\varnothing$. The family $\mathcal{T}$ of ordered rooted trees can be identified with the set of all subtrees $\T$ of $\U$ that have the property that for all $v\in V_{\T}$:
$$vi\in V_{\T}\ \quad  \text{implies}\quad  \ vj\in V_{\T}, \quad \text{for all }j\leq i.$$
Above $V_{\T}$ is the vertex set of $\T$, and we identify it with the tree $\T$ itself when no confusion arises.
For a tree in $\mathcal{T}$  which is not rooted at $\varnothing$, but at some other vertex $u\in \U$, we write $\T_u$. For two vertices $u,v\in \U$ we denote by $d(u,v)$ their graph distance, that is, the length of the shortest path between $u$ and $v$.
For trees rooted at $\varnothing$, we omit the root and we write only $\T$.

For $J\in \N$, a $J$-type tree is a pair $(\T,\type_{\T})$ where $\T\in \mathcal{T}$ and $\type_{\T}:\T\to\{1,\ldots,J\}$ is a function defined on the vertices of $\T$ which returns for each vertex $v$ its type $\type_{\T}(v)$. We will often omit the subindex $\T$ in the type function $\type$ if it is clear from the context the tree we are referring to. We denote by $\mathcal{T}^{[J]}$ the set of all $J$-type trees, and elements of $\mathcal{T}^{[J]}$ will be referred to as  $\T$ without explicitly mentioning the type function $\type$.

\textbf{Multitype Galton-Watson Processes.}
Galton-Watson trees are an important class of random rooted trees that are defined as the family trees of Galton-Watson processes. Having introduced the Ulam-Harris tree above, we can now regard Galton-Watson processes as $\mathcal{T}$-valued random variables.
Let $J\in \N$ be the number of types of particles, and  consider  a sequence  $(L^{(j)})_{j\in [J]}$ of $\N^{J}$-valued, $J$ independent random (column) vectors, and denote by $L$ the $J\times J$ random matrix with column vectors $L^{(j)}$:
\begin{small}$$L=\big(L^{(1)},L^{(2)},\ldots,L^{(J)}\big)
=\begin{pmatrix}
   L^{(1,1)}  & L^{(2,1)} &  \cdots & \cdots  &  L^{(J,1)}\\
   L^{(1,2)}  & L^{(2,2)} & \cdots  & \cdots  & L^{(J,2)} \\
    \vdots    & \vdots    & \vdots &  \vdots & \vdots \\
   L^{(1,J)}   & L^{(2,J)}  & \cdots & \cdots   & L^{(J,J)}
   \end{pmatrix},
  $$\end{small}where for $j\in [J]$ the column vector $L^{(j)}$ is given by $L^{(j)}=\left(L^{(j,1)},\ldots,L^{(j,J)}\right)$. 
We think of $L^{(j)}=L\mathrm{e}_j$ as the offspring vector produced by a particle of type $j$. More precisely, the  matrix $L=(L_{ij})_{i,j\in [J]}$ with $L_{ij}=L^{(j,i)}$ represents the offspring distribution matrix, so for  $i,j\in [J]$, $L^{(j,i)}$ is the number of offspring of type $i$ produced by a particle of type $j$ (here $L^{(j,i_1)}$ and $L^{(j,i_2)}$ might be dependent).
We denote by $a_{ij}=\E\left[L_{ij}\right]$ the expectation of the random variable $L_{ij}$, for all $i,j\in [J]$ and by $A=(a_{ij})_{i,j\in [J]}$ the first moment  or the mean offspring matrix. We have  $\E L=A$.

Note our choice of notation; it may be more natural to use the transpose of the matrix $L$ and its componentwise mean matrix $A$, as  done in the relevant literature on multitype branching processes. The reason for our choice is that we need the standard notation where a matrix is regarded as an operator acting on column vectors to the right, as opposed to the Markov chains notation, where the transition matrix acts on row vectors to the left. Even if we work with the transpose $L$ of the offspring distribution matrix  $L^{\top}$, we will still, by a slight abuse of notation, refer to $L$ and $A$ as the \textit{offspring distribution matrix} and the \textit{mean offspring matrix}, respectively.

The matrix $A$ is called positively regular, if for every $i, j ∈ [J]$, there is
some $n\in \N$ so that, the $(i,j)$-entry  of $A^n$ is positive.
With the assumption of A being positive regular, Perron–Frobenius theorem ensures that the eigenvalue $\rho$ of $A$ with maximum modulus is real, positive, and simple, and
there exist left and right eigenvectors of $A$, denoted by  $\mathrm{v}=(\mathrm{v}_1,\ldots,\mathrm{v}_J)$ and $\mathrm{u}=(\mathrm{u}_1,\ldots,\mathrm{u}_J)$ respectively, such that
\begin{align*}
&\mathrm{v}^{\top}A=\rho \mathrm{v}^{\top},\quad \mathrm{v}_j>0,\quad 1\leq j\leq J\\
&A\mathrm{u}=\rho \mathrm{u},\quad \mathrm{u}_j>0,\quad 1\leq j\leq J\\
&\langle \mathrm{u},\mathrm{v}\rangle=\sum_{j=1}^J\mathrm{u}_j\mathrm{v}_j=1,
\end{align*}
so $\mathrm{u}$ and $\mathrm{v}$ are normalized such that their inner product is $1$. Remark that if $\mathrm{u}$ is a right eigenvector for $A$, then $\mathrm{u}^{\top}$ is a left eigenvector for $A^{\top}$.

With the random variables $L^{(i,j)}$, $i,j\in[J]$, as introduced above, we define
multitype Galton-Watson trees  as $\mathcal{T}^{[J]}$-valued random variables, where the type function $\type$ is random and defined in terms of i.i.d. copies of  the $J\times J$-random matrix $L$ as following. Let $(L(u))_{u\in \U}$ be a family of i.i.d copies of $L$ indexed after the vertices of $\U$.
For any $i\in [J]$, we define the random labeled tree $\T^{i}\in \mathcal{T}^{[J]}$ rooted at $\varnothing$, with the associated type function
$\type=\type^{i}:\T^{i}\to \{1,\ldots,J\}$ defined 
recursively as follows:
$$\varnothing\in\T^{i}\quad\text{and}\quad \type(\varnothing)=i.$$
Now suppose that $u=u_1\ldots u_m\in\T^{i}$ with $\type(u)=j$, for some $j\in[J]$. Then 
$$u_1 \ldots u_m k\in \T^{i} \quad\text{iff}\quad k\le L^{(j,1)}(u)+\dots+L^{(j,J)}(u)$$ and we set
$\mathrm \type(u_1 \ldots u_m k)=l$ whenever $$L^{(j,1)}(u)+\dots+L^{(j,l-1)}(u)<k\le L^{(j,1)}(u)+\dots+L^{(j,l)}(u).$$
Note that $L(u)e_{\type(u)}$ is the vector of children of the individual $u$.
The multitype branching process  $Z_n=(Z^1_n, \dots , Z^J_n)$ associated with $(\T^{i_0},\type)$ and starting from a single particle of type $i_0\in [ J]$ at the root $\varnothing$, that is $\type(\varnothing)=i_0$, is defined as: $Z_0=\mathrm{e}_{i_0} $ and for $n\geq 1$
\begin{align*}
Z^i_n\defeq\#\{u\in\T^{i_0}:|u|=n\text{ and }\type(u)=i\}, \quad \text{for } i\in [J],
\end{align*}
that is, $Z_n^i$ represents the number of particles of type $i$ in the $n$-th generation, or the number of vertices $u\in \T^{i_0}$ with $|u|=n$ and $\type(u)=i$. 
For the filtration $(\F_n)_{n\ge0}$ defined by $\F_n\defeq\sigma(\{L(u):|u|\le n\})$, it is easy to verify that 
\begin{align}
\label{eq:conditional relation}
\E[Z_{n+1}|\F_n]=A Z_n.
\end{align}
When referring to multitype Galton-Watson processes, we shall always have in mind both $(Z_n)_{n\in \N}$ and its genealogical tree $\T$, with the corresponding type function $\type$.
\paragraph{Assumptions.}
During the current work, we shall make the following assumptions for the multitype Galton-Watson process $(Z_n)_{n\in \N}$.
\vspace{-0.25cm}
\begin{enumerate}[(GW1)]
\setlength\itemsep{0em}
\item $(Z_n)_{n\in \N}$ is supercritical, that is, $\rho>1$.
\item The first moment matrix $A$ is positively regular.
\item $\mathbf{0}\neq \sum_{j=1}^J\Cov\left[L^{(j)}\right]$ and all the variances $\Var[L^{(i,j)}]$ are finite. 
\end{enumerate}
Above,  $\Cov\left[L^{(j)}\right]$ is the covariance matrix of the random vector $L^{(j)}$, and $\mathbf{0}$ is the zero matrix of size $J\times J$. We recall that under assumptions (GW1)-(GW3), Kesten-Stigum theorem \cite{kesten1966} ensures the existence of a scalar random variable $W$ such that $\rho^{-n}Z_n\to W\mathrm{u}$ almost surely as $n\to \infty$. For the rest of the paper, when we use the random variable $W$, we always mean the limit random variable from the Kesten-Stigum theorem.

\textbf{Spectral decomposition of the matrix $A$.}
The spectral decomposition of the matrix $A$ will play an important role in the proofs. We denote by
$\sigma_A$ the spectrum of the matrix $A$
and  we split it as $\sigma_A=\sigma_A^{1}\cup\sigma_A^{2}\cup\sigma_A^{3}$ according to whether for $\lambda\in\sigma_A$, $|\lambda|$ is greater, equal or smaller than $\sqrt \rho$, respectively. 
From the Jordan-Chevalley decomposition (which is unique up to the order of the Jordan blocks) of $A$ we infer the existence of projections $(\pi_\lambda)_{\lambda\in \sigma_A}$
that commute with $A$ and satisfy $\sum_{\lambda\in\sigma_A}\pi_{\lambda}=I$
and 
$$A\pi_{\lambda}=\pi_{\lambda}A=\lambda\pi_{\lambda}+N_{\lambda}$$
where $N_{\lambda}=\pi_{\lambda}N_{\lambda}=N_{\lambda}\pi_{\lambda}$ is a nilpotent matrix. Moreover,
for any $\lambda_1,\lambda_2\in\sigma_A$ it holds  $\pi_{\lambda_1}\pi_{\lambda_2}=\pi_{\lambda_1}\ind{\lambda_1=\lambda_2}$. If $\lambda\in\sigma_A$ is a simple eigenvalue of $A$ and $\mathrm u_\lambda, \mathrm v_\lambda$ are the corresponding left and right eigenvectors normalized in such a way that $\mathrm v_\lambda\mathrm u_\lambda=1$  then 
$\pi_\lambda=\mathrm u_\lambda\mathrm v_\lambda$.
 If we write $N=\sum_{\lambda\in\sigma_A}N_{\lambda}$, then $N$ is also a nilpotent matrix and we have $N\pi_{\lambda}=N_{\lambda}$. So $A$ can be decomposed in the semisimple $D:=\sum_{\lambda\in\sigma_A}\lambda\pi_{\lambda}$ and the nilpotent part $N$ as:
$A=D+N$.
We define $\pi^{(1)},\pi^{(2)},\pi^{(3)}$ as following:
\begin{align*}
\pi^{(1)}\defeq\sum_{\lambda\in\sigma^1_A}\pi_\lambda,\quad
\pi^{(2)}\defeq\sum_{\lambda\in\sigma^2_A}\pi_\lambda,\quad
\pi^{(3)}\defeq\sum_{\lambda\in\sigma^3_A}\pi_\lambda.
\end{align*}
Clearly $\pi^{(1)}+\pi^{(2)}+\pi^{(3)}=I$, where $I$ is the $J\times J$ identity matrix. By $V^{(i)}$, for $i=1,2,3$ we denote the corresponding  images of the projection $\pi^{(i)}$. Thus in $V^{(1)}$, $V^{(2)}$ and $V^{(3)}$ we glob together the Jordan block invariant subspaces corresponding to the eigenvalues $\lambda$ that satisfy $|\lambda|>\sqrt{\rho}$, $|\lambda|=\sqrt{\rho}$, $|\lambda|<\sqrt{\rho}$ respectively.
Note that, as $V^{(i)}$, $i=1,2,3$ are orthogonal, we may write $\C^J$ as a direct sum $\C^J\defeq V=V^{(1)}\oplus V^{(2)}\oplus V^{(3)}$.  
Finally, for  $i=1,2$ we set 
$$A_i\defeq A\pi^{(i)}+\big(I-\pi^{(i)}\big),$$ 
i.e. $A_i$ acts as $A$ on $V^{(i)}$ and as identity on its orthogonal complement.
Clearly, $A_1$ and $A_2$ are both invertible. 

\begin{remark}
Taking  \eqref{eq:conditional relation} into account, one can easily check that for $i=1,2$ the process  $\big(W_n^{(i)}\big)_{n\in \N}$ defined by 
$$W_n^{(i)}\defeq A_i^{-n}\pi^{(i)} Z_n$$ is a $V^{(i)}$-valued $\F_n$-martingale.  
\end{remark}
\begin{lemma}
	\label{lem:convergence of martingales}
	Suppose that (GW1)-(GW3) hold. Then  the martingale $W^{(1)}_n$ converges in $\mathcal{L}^2$ to some random variable denoted $W^{(1)}$, as $n\to\infty$.
\end{lemma}
We postpone the proof to Section \ref{sec:martingales}.
Notice that  with our notation, the limit $W\mathrm{u}$  from Kesten-Stigum theorem is the projection of $W^{(1)}$ on $\mathrm {u}$ i.e. $\pi_{\rho}W^{(1)}=W\mathrm {u}$.

\section{Branching processes counted with characteristics}
\label{sec:branch-char}

We now introduce our main object of study, a discrete time stochastic process that we call  \emph{branching processes counted with a characteristic} or simply \emph{branching processes with a characteristic}. 
\begin{definition}
Let $d\in\N$ be fixed. A \emph{characteristic} of  dimension $d$ is a random function $\Phi:\Z\to \C^{d\times J}$,  that is, for each $k\in \Z$, $\Phi(k)$ is a $\C^{d\times J}$-valued random variable defined on the same probability space where the branching process is defined.
\end{definition}
By a {\it deterministic characteristic $\Phi$} we mean just a fixed doubly-infinite sequence $(\Phi(k))_{k\in \Z}$ of matrices in $\C^{d\times J}$, i.e. a deterministic function defined on $\Z$.
The $j$-th column vector of $\Phi(k)$ is then
$\Phi(k)\mathrm{e}_j\in \C^{d}$, for every $k\in\Z$ and $j\in[J]$.
We assume that for any $k\in\Z$, $\E\left[\|\Phi(k)\|\right]<\infty$ and for any $u\in\U$ there is a $\Phi_u$ such that 
$$\left((L(u),\Phi_u)\right)_{u\in\U}$$ is a collection of i.i.d.  copies of the pair $(L,\Phi)$. That is, any vertex $u\in \U$ is equipped with its own pair $(L(u),\Phi_u)$ of offspring distribution matrix $L(u)$ and the sequence $(\Phi_u(k))_{k\in \Z}$ of $d\times J$ random matrices, respectively, and we interpret $\Phi_u(k)$ as the characteristic of $u$ at age $k$, for $k\in\Z$.   If $\Phi$ is a deterministic function of the offspring distribution matrix $L$, then the above condition is automatically satisfied.  A characteristic of dimension $d=1$ is a function $\Phi:\Z\to \C^{1\times J}$, i.e.  each $\Phi(k)$ is a row vector with $J$ entries.

\begin{definition} \label{def:br-char}Let $(\T,\type)\in \mathcal{T}^{[J]}$ be a  multitype Galton-Watson tree with $J$ types. To the pair $(\T,\type)$ and the characteristic $\Phi:\Z\to \C^{d\times J}$ we associate the stochastic process $(Z_n^{\Phi})_{n\in \N}$, defined as: for every $n\in \N$
\begin{align}
	\label{eq:definition of CMJ}
Z^\Phi_n=\sum_{u\in\T}\Phi_u(n-|u|)\mathrm{e}_{\type(u)}.
\end{align}
We call $(Z^\Phi_n)_{n\in \N}$  a branching process counted with characteristic $\Phi$.
\end{definition}
Heuristically, the branching process $(Z^\Phi_n)_{n\in \N}$  with the (random or deterministic)
characteristic $\Phi$ counts all the individuals with the corresponding types in the multitype Galton-Watson tree $(\T,\type)$,
and the contribution of each individual to the sum is determined by the characteristic $\Phi$. This characteristic may take into account some aspects of the individual such as age, health, siblings, etc.
For every $n\in \N$,  $Z_n^{\Phi}$ is a vector in $\C^d$.
For a deterministic characteristic $\Phi$, i.e. a fixed function, we have  for every $n\in\N$
\begin{align*}
Z^\Phi_n=\sum_{k=0}^{\infty}\Phi(n-k)Z_{k}=\sum_{k=-\infty}^n\Phi(k)Z_{n-k}.
\end{align*} 
\begin{example} A trivial example of  a deterministic characteristic of dimension $J$ is 
$$\Phi(k)=\ind{k=0}I,$$
where $I$ is the $J\times J$ identity matrix. In this case we recover the multitype Galton-Watson process, that is $Z_n^{\Phi}=Z_n$, for all $n\in\N$.
\end{example}
\paragraph{Assumptions on the characteristic $\Phi$.}
Our standing assumptions on the characteristic $\Phi$ during the current work are the following: 
\begin{equation}\label{eq:main-assum1}
\sum_{k\in \Z}\big\|\E[\Phi(k)]\big\|(\rho^{-k}+\vartheta^{-k})<\infty,\tag{CH1}
\end{equation}
for some $\vartheta<\sqrt{\rho}$ and
\begin{equation}\label{eq:main-assum2}
\sum_{k\in \Z}\|\Var[\Phi(k)]\|\rho^{-k}<\infty.\tag{CH2}
\end{equation}
Note that for an arbitrary characteristic $\Phi$, the convergence of the series in \eqref{eq:definition of CMJ} is not obvious. We comment first on why $(Z_n^{\Phi})_n$ given by \eqref{eq:definition of CMJ} is well defined.
\begin{proposition}
Suppose that the random characteristic $\Phi$ satisfies \eqref{eq:main-assum1} and \eqref{eq:main-assum2}. Then the infinite series
$$\sum_{u\in\U}\ind{u\in\T}\Phi_u(n-|u|)\mathrm{e}_{\type(u)}$$
converges unconditionally  in $\L^1$ (the limit does not depend on the ordering of vertices in $\U$). In particular, the process $(Z^\Phi_n)_{n\in \N}$ given by \eqref{eq:definition of CMJ} is well-defined.
\end{proposition}	
\begin{proof}
	We have
\begin{align*}
	\E\Big[\sum_{u\in\T}\|\E\Phi_u(n-|u|)\|\Big]=\sum_{k= 0}^{\infty}\E\left[\|\Phi(n-k)\|\right]\langle{\bf 1},A^kZ_0\rangle<\infty,
\end{align*} 
that is, the series defining $Z^{\E\Phi}_n$ converges absolutely and so  unconditionally in $\L^1$. On the other hand, if for a finite $U\subset\U$ we set 
$$Z(U)\defeq\sum_{u\in\T\cap U}\big(\Phi_u(n-|u|)-\E\Phi_u(n-|u|)\big)\mathrm{e}_{\type(u)}$$
then for $U\subseteq V\subset\U$, as the cross-terms vanish, we have 
\begin{align*}
	\E\Big[&\big(Z(V)-Z(U)\big)^2\Big]=\E\Big[\sum_{u\in\U}\ind{u\in\T\cap (V\setminus U)}\Var[\Phi(n-|u|)]\mathrm{e}_{\type(u)}\Big]\\
	&\le\E\Big[\sum_{u\in\U}\ind{u\in\T}\Var[\Phi(n-|u|)]\mathrm{e}_{\type(u)}\Big] 
	= \rho^n\sum_{k\le n}\|\Var[\Phi(k)]\|\rho^{-k}<\infty,
\end{align*}
and so for any increasing sequence $U_n\nearrow\U$, by the dominated convergence theorem $Z(U_n)$ is a $\L^2$-Cauchy sequence  and so it converges. This justifies that the series defining $Z_n^{\Phi-\E\Phi}$ converges unconditionally in $\L^2$ and so in $\L^1$ as well.
\end{proof}

For a given characteristic $\Phi$ of dimension $d$ that fulfills \eqref{eq:main-assum1} and \eqref{eq:main-assum2} we finally define two $d\times J$ matrices $\mathrm{x}_1$ and $\mathrm{x}_2$ by
\begin{equation}\label{eq:vect-x1-x2}
\mathrm{x}_1=\mathrm{x}_1(\Phi)\defeq\sum_{k\in \Z}\E[\Phi(k)]\pi^{(1)}A_1^{-k}\quad \text{and} \quad \mathrm{x}_2=\mathrm{x}_2(\Phi)\defeq\sum_{k\in \Z}\E[\Phi(k)]\pi^{(2)}A_2^{-k},
\end{equation}
and for  $l=0,\dots,J-1$, the constants
\begin{align}
\label{eq:definition of sigma_l}
\sigma_l^2\defeq \frac{\rho^{-l}}{(2l+1)(l!)^2}\sum_{\lambda\in\sigma^2_A}\Var \Big[\mathrm{x} _2 \pi_\lambda (A-\lambda I)^{l}L\Big]\mathrm{u}.
\end{align}
Note that if $l$ is larger than the size of the largest Jordan block corresponding to  eigenvalues $\lambda\in \sigma_A^2$ then $(A-\lambda I)^{l}=0$ and therefore  $\sigma_l$ is also $0$. On the other hand, it may happen that for all eigenvalues $\lambda$ on the critical  circle $\{|z|=\sqrt \rho\}$, we have 
$$\mathrm{x} _2 \pi_\lambda (A-\lambda I)^{l}L=\mathrm{x} _2 \pi_\lambda (A-\lambda I)^{l}A\neq0$$
i.e. the matrix $L$ is deterministic in some direction, and this  also yields  $\sigma_l=0$. 
Our main result is the following; below we write $\stabto$ for stable convergence.
\begin{theorem}\label{thm:main}
Assume $(GW1)-(GW3)$ hold and let $\Phi:\Z\to\C^{1\times J}$ be a  random characteristic of dimension one that satisfies \eqref{eq:main-assum1} and \eqref{eq:main-assum2}. Then for a standard normal random variable $\mathcal{N}(0,1)$ independent of $W$, the following stable convergence holds.
\vspace{-0.25cm}
\begin{enumerate}[i)]
\setlength\itemsep{0em}
\item If $\sigma_l=0$ for all $0\le l\leq J-1$,  then there exists a constant $\sigma\ge0$ such that either $\sigma>0$ and
\begin{align*}
\frac{Z_n^{\Phi}-\mathrm{x_1}A^n_1 W^{(1)}-\mathrm{x_2}A^n_2Z_0 }{\rho^{n/2}}\stabto \sigma\sqrt W\cdot\mathcal{N}(0,1),\quad \text{as } n\to\infty
\end{align*} 
or $\sigma=0$ and the left hand side above is a deterministic sequence converging to 0.
\item Otherwise, let $0\le l\le J-1$ be maximal such that $\sigma_l\neq0$. Then
\begin{align*}
\frac{Z_n^{\Phi}-\mathrm{x_1}A_1^n W^{(1)}-\mathrm{x_2}A_2^nZ_0 }{n^{l+\frac 12}\rho^{n/2}}\stabto \sigma_l\sqrt W\cdot\mathcal{N}(0,1), \quad \text{as } n\to\infty.
\end{align*}
	\end{enumerate}
\end{theorem}
\begin{remark}
	The constant $\sigma$ can be explicitly computed and is given by \eqref{eq:definition of sigma}, even though the calculations are pretty involved.
\end{remark}
\begin{remark}
	Since $Z_n^\Phi, \mathrm{x}_1(\Phi)$ and $\mathrm{x}_2(\Phi)$ depend linearly on the characteristic $\Phi$, by the Cram\'er-Wold device we can immediately conclude that for different characteristics $\Phi_1,\dots,\Phi_k$ the convergence holds jointly. In particular, Theorem \ref{thm:main} can be easily extended to multidimensional $d>1$ characteristics. 
\end{remark} 

\section{Auxiliary results}\label{sec:aux-results}

We now build towards proving Theorem~\ref{thm:main}. For this, we first prove several  preliminary results that we finally put together in order to complete the main proof. The three projections $\pi^{(1)},\pi^{(2)},\pi^{(3)}$ and the martingales $W_n^{(1)}$ and $W_n^{(2)}$ will be investigated separately.
While the main result claims limit theorems for general  random characteristics $\Phi$,  our proof will be dealt with in the following increasing order of difficulty:
\begin{enumerate}[i)]
\setlength\itemsep{0em}
\item Leading order of $Z^\Phi_n$, when $\Phi$ is a deterministic characteristic.
\item Limit theorem of $Z^\Phi_n$, when  $\Phi$ is a random and centered characteristic.
\item Reduction of $Z^\Phi_n$  to the previous case, when $\Phi$ is a deterministic characteristic with small growth of $\E Z^\Phi_n$.
\end{enumerate}

\subsection{Limit theorem for deterministic characteristics}

For a deterministic function $\Phi:\Z\to\C^{1\times J}$, the next result gives a strong law of large numbers for $(Z_n^{\Phi})_{n\in\N}$, and this can be seen as discrete time analogue of \cite[Theorem 7.2, Theorem 7.3]{jagers}, with slightly stronger assumptions on the characteristic. While the proof is an easy exercise that can be deduced from \cite{jagers,iksanov-meiners}, we still include it here for sake of completeness, since we are not aware of such a result in the context of discrete time, multitype CMJ processes counted with some characteristic $\Phi$.
For a deterministic characteristic $\Phi$, \eqref{eq:main-assum2} is trivially satisfied.
\begin{proposition}\label{prop:appl-kesten}
Under assumptions $(GW1)-(GW3)$, let $\Phi:\Z\to\C^{1\times J}$ be a  deterministic characteristic and $(Z_n^{\Phi})_{n\in\N}$ the associated branching process counted with characteristic $\Phi$. If $W$ is the  scalar random variable from Kesten-Stigum theorem, then the following holds.
\begin{enumerate}[i)]
\setlength\itemsep{0em}
\item If $\sum_{k\in \Z}\|\Phi(k)\|\rho^{-k}<\infty$, then for $\sigma^2_\Phi=\sum_{k\in \Z}\Phi(k)\rho^{-k}\mathrm{u}$ we have
\begin{align}\label{eq: nerman's type convergence}
\lim_{n\to\infty}\frac{Z^\Phi_n}{\rho^{n}}=\sigma^2_\Phi W\quad\textit{almost surely}.
\end{align}
\item If for some $\gamma\ge0$ and some  $\sigma^2\in\C$ the following three conditions hold: 
\begin{align*}
\sup_{n\ge 1} \frac{1}{n^{\gamma}}\sum_{k\in \Z}\|\Phi(k)\|\rho^{-k}<\infty,
\end{align*}
and 
\begin{align*}			
\lim_{n\to\infty}\frac{1}{n^{\gamma}\rho^{n}}\|\Phi(n)\|= 0\quad \text{and}\quad 
\lim_{n\to\infty}\frac{1}{n^{\gamma}}\sum_{k=-\infty}^n\Phi(k)\rho^{-k}\mathrm{u}=\sigma^2,
\end{align*}
then  
\begin{align}
\label{eq: nerman's type convergence II}
\lim_{n\to\infty}\frac{Z^\Phi_n}{n^{\gamma}\rho^{n}}=\sigma^2W\quad\textit{almost surely.}
\end{align}
\end{enumerate}
\end{proposition}
\begin{proof}
{\it i).} Since
$$\frac{1}{\rho^{n}}Z^\Phi_n=\sum_{k=-\infty}^n\Phi(k)\rho^{-k}\rho^{-n+k}Z_{n-k},$$
together with Kesten-Stigum theorem and dominated convergence, 
we obtain \eqref{eq: nerman's type convergence}.

{\it ii).} In order to prove {\it ii)}, fix $\epsilon>0$. Then, again by Kesten-Stigum theorem  there exists (random) $k_0$ such that for any $k\ge k_0$, $\|W\mathrm{u}-\rho^{-k}Z_k\|<\epsilon$. Then
	\begin{align*}
	\limsup_{n\to \infty}&\|n^{-\gamma}\rho^{-n}Z^\Phi_n-\sigma^2 W\|
	=\limsup_{n\to \infty}\Big\|n^{-\gamma}\sum_{k=-\infty}^n\Phi(k)\rho^{-k}\big(\rho^{-n+k}Z_{n-k}-W\mathrm{u}\big)\Big\|\\
	&=\limsup_{n\to \infty}\Big\|n^{-\gamma}\sum_{k=-\infty}^{n-k_0}\Phi(k)\rho^{-k}\big(\rho^{-n+k}Z_{n-k}-W\mathrm{u}\big)\Big\|\\
	&\quad+\limsup_{n\to \infty}\Big\|n^{-\gamma}\sum_{k=n-k_0+1}^{ n}\Phi(k)\rho^{-k}\big(\rho^{-n+k}Z_{n-k}-W\mathrm{u}\big)\big\|\\
	&\le \epsilon\sup_{n\ge 1}\big( n^{-\gamma}\sum_{k=-\infty}^n\|\Phi(k)\|\rho^{-k}\big)
+\left(\|Z_{k_0}\|+W\|\mathrm{u}\|\right) \limsup_{n\to \infty}n^{-\gamma}\sum_{k=n-k_0+1}^ n\|\Phi(k)\|\rho^{-k}\\
	&= \epsilon\sup_{n\ge 1}\big( n^{-\gamma}\sum_{k\le n}\|\Phi(k)\|\rho^{-k}\big),
	\end{align*}
and the last term above is finite by assumption, so this 	proves \eqref{eq: nerman's type convergence II}, since $\epsilon$ was arbitrary small. 
\end{proof}

\subsection{Central limit theorem for centered characteristics}

We call a random characteristic $\Phi:\Z\to \C^{1 \times J}$ {\it centered} if for any $k\in\Z$ it holds $\E[\Phi(k)]=(0\ 0\ldots0):=\mathbf{0}\in \R^{1\times J}$. The next result gives a central limit theorem for branching processes $(Z_n^{\Phi})_{n\in \N}$ counted with a centered characteristic $\Phi$.
\begin{theorem}\label{thm: centred characteristic}
Assume $(GW1)-(GW3)$ hold and let $\Phi:\Z\to\C^{1\times J}$ be a  centered random characteristic of dimension one with associated branching process $(Z_n^{\Phi})_{n\in\N}$. For a standard normal random variable  $\mathcal{N}(0,1)$  independent of $W$, the following holds.
\vspace{-0.25cm}
\begin{enumerate}[i)]
\setlength\itemsep{0em}
\item If 
$\sum_{k\in\Z}\|\Var[\Phi(k)]\|\rho^{-k}<\infty$, then for $\sigma^2_\Phi= \sum_{k\in \Z}\Var\left[\Phi(k)\right]\rho^{-k}\mathrm{u}$ we have
$$\frac{Z^\Phi_n}{\rho^{n/2}}\stabto \sigma_\Phi \sqrt W\cdot\mathcal{N}(0,1),\quad \text{as }n\to\infty.$$
Moreover,  $\sigma_\Phi=0$ if and only if $Z^\Phi_n=0$ for all $n\in\N$.
\item If for some constants $\gamma\ge0,\sigma>0$ the following two conditions hold:
\begin{align*}
&\sup_{n\ge 1}n^{-\gamma}\sum_{k=-\infty}^n\|\Var[\Phi(k)]\|\rho^{-k}<\infty\\
&\lim_{n\to\infty}n^{-\gamma}\sum_{k=-\infty}^n\Var[\Phi(k)]\rho^{-k}\mathrm{u}=\sigma^2,
\end{align*}
then
$$\frac{Z^\Phi_n}{n^{\gamma/2}\rho^{n/2}}\stabto\sigma\sqrt W\cdot\mathcal{N}(0,1),\quad \text{as }n\to\infty.$$
\end{enumerate}
\end{theorem}
\begin{proof}
For simplicity of notation, we write $V_\Phi$ for the variance of $\Phi$ during this proof, that is for $k\in \Z$, $V_{\Phi}(k)\defeq \Var[\Phi(k)]$  and for any $(k,i)\in\Z\times [J]$ we define the deterministic characteristic $\Psi:\Z\to \C^{1\times J}$ by
	$$\Psi(k)\mathrm{e}_j\defeq\E[(\Phi(k)\mathrm{e}_j)^2].$$
We consider an increasing sequence $(G_n)_{n\in\N}$ of subsets of $\U$ with the following property:
$\cup_{n\ge 1} G_n=\U$, for any $n\in\N$, $|G_n|=n$ and finally
if $u\in G_n$ then for any $v\le u$, $v\in G_n$.
Such a sequence can be constructed using the diagonal method. By $v_n$ we denote the unique vertex in $G_n\setminus G_{n-1}$ and we let $\cG_n\defeq\sigma(\{L(u):u\in G_n\})$. We define a normalizing sequence $(r_n)_{n\ge1}$ by
	\begin{align*}
	r_n\defeq
	\begin{cases*}
	\rho^{n/2}&\quad\textit{in case i)}\\
	n^{\gamma/2}\rho^{n/2} &\quad\textit{in case ii)}
	\end{cases*}
	\end{align*}	
For any $n\in\N$, the process $(M_k(n))_{k\ge1}$ defined by 
$$M_k(n)\defeq r_n^{-1}\sum_{u \in G_k} \Phi_u(n-|u|)\mathrm{e}_{\type(u)}$$
is a $\cG_k$--martingale. Indeed, for any $u\in G_k$ both $\type(u)$ and $\Phi_u$ are $\cG_k$--measurable  and so is $M_k(n)$. The martingale property then follows since $\type(v_{k+1})$ is $\cG_k$--measurable and $\Phi_{v_{k+1}}$ is $\sigma\left(L(v_{k+1})\right)$--measurable, whereas $\Phi_{v_{k+1}}$ is independent of $\cG_k$, hence
$$\E\left[\Phi_{v_{k+1}}(n-|v_{k+1}|)\mathrm{e}_{\type(v_{k+1})}|\cG_k\right]=0,$$
since $\Phi$ is a centered characteristic by assumption. Next, using again the fact that $\Phi$ is centered and for $u\neq v$, $\Phi_u$ and $\Phi_v$ are independent, together with the assumptions we get
\begin{align*}
\E\left[M_k(n)^2\right]=r_n^{-2}\E\Big[\sum_{u\in G_k}\E[(\Phi_u(n-|u|)\mathrm{e}_{\type(u)})^2]\Big]\leq r_n^{-2} \E\big[Z_n^{\Psi}\big]\leq C,
\end{align*}
for some constant $C\geq 0$ independent of $n$ and $k$, so the martingale $(M_k(n))_k$ is bounded in $\L^2$ and hence it  converges almost surely and in $\L^2$ to some limit random variable $M(n)$.	
Let now $(n_k)_{k\in \N}$ be  an increasing sequence in $\N$. 
Then there exists an increasing subsequence $(i_k)_{k\in \N}$ such that $\E[(M(n_k)-M_{i_k}(n_k))^2]\le 2^{-k}$ and hence  $M(n_k)-M_{i_k}(n_k)$ converges to 0 almost surely as $k\to\infty$. In view of the martingale central limit theorem   \cite[Corollary 3.1 on p.~58]{Hall+Heyde:1980} it suffices
to verify in case \textit{i)}, with $r_n=\rho^{n/2}$ that 
\begin{align}
\label{eq:mgale CLT1}
&r_{n_k}^{-2}\sum_{i=1}^{i_k}\E\Big[\left( \Phi_{v_i}(n_k-|v_i|)\mathrm{e}_{\type(v_i)}\right)^2\Big|{\cG_{i-1}}\Big]\Probto
\sigma^2_\Phi W,\quad k\to\infty\\
\label{eq:mgale CLT2}
&r_{n_k}^{-2}\sum_{i=1}^{i_k}\E\Big[ \left(\Phi_{v_i}(n_k-|v_i|)\mathrm{e}_{\type(v_i)}\right)^2\ind{|\Phi_{v_i}(n_k-|v_i|)\mathrm{e}_{\type(v_i)}|>\delta r_{n_k}}\Big|\cG_{i-1}\Big]
\Probto 0,\quad k\to\infty
\end{align}
for every $\delta>0$, while in case {\it ii)} the same two convergence results should hold for $r_n=n^{\gamma/2} \rho^{n/2}$, with $\sigma_{\Phi}^2$ being replaced with $\sigma^2$ in equation \eqref{eq:mgale CLT1}. Observe first that
\begin{equation*}
r_{n_k}^{-2}\E \bigg[\sum_{i=i_k+1}^{\infty} \E\Big[\left( \Phi_{v_i}(n_k-|v_i|)\mathrm{e}_{\type(v_i)}\right)^2\big|{\cG_{i-1}}\Big]\bigg]<\E\big[(M(n_k)-M_{i_k}(n_k))^2\big]\leq 2^{-k},
\end{equation*}
and 
\begin{align*}
r_{n_k}^{-2}\sum_{i=1}^{i_k}\E\big[ \left(\Phi_{v_i}(n_k-|v_i|)\mathrm{e}_{\type(v_i)}\right)^2\big|&{\cG_{i-1}}\big]=r_{n_k}^{-2}\sum_{i=1}^{i_k} \Psi(n_k-|v_i|)\mathrm{e}_{\type(v_i)},
\end{align*}
hence \eqref{eq:mgale CLT1} reduces to proving that
$$r_{n_k}^{-2}Z_{n_k}^{\Psi}=r_{n_k}^{-2}Z_{n_k}^{\Var[\Phi]}\Probto
\sigma^2_\Phi W,\quad \text{ as } k\to\infty$$
in case {\it i)}, with $\sigma_{\Phi}^2$ being replaced with $\sigma^2$ in case \textit{ii)}.
Applying Proposition \ref{prop:appl-kesten} {\it i)}	
to the one dimensional, deterministic characteristic $\Var[\Phi(k)]$, for 
$\sigma^2_\Phi=\sum_{k\in \Z}\Var[\Phi(k)]\rho^{-k}\mathrm{u}$ equation \eqref{eq: nerman's type convergence} yields
$$\rho^{-n_k}Z_{n_k}^{\Var[\Phi]} \to \sigma^2_\Phi W ,\quad\text{almost surely as }k\to\infty,$$
and this proves \eqref{eq:mgale CLT1} in case {\it i)} of the claim.
On the other hand, Proposition \ref{prop:appl-kesten} {\it ii)} applied to the same characteristic $\Var[\Phi(k)]$ and
$r_n=n^{\gamma/2}\rho^{n/2}$, together with \eqref{eq: nerman's type convergence II} yields for \\
$\lim_{n\to\infty}n^{-\gamma}\sum_{k=-\infty}^n\Var[\Phi(k)]\rho^{-k}\mathrm{u}=\sigma^2$ that
$$\rho^{-n_k}n_k^{-\gamma}Z_{n_k}^{\Var[\Phi]} \to \sigma^2 W ,\quad\text{almost surely as }k\to\infty,$$
thus proving \eqref{eq:mgale CLT1} in case {\it ii)} of the claim.	

In order to prove  \eqref{eq:mgale CLT2}, for any $T>0$, we consider the truncated characteristic
	$$\Psi^T(n)\mathrm{e}_i\defeq\E\big[(\Phi(n)\mathrm{e}_i)^2\ind{|\Phi(n)\mathrm{e}_i|>T}\big]$$ and observe that $\Psi^T(n)\mathrm{e}_i\le \Psi(n)\mathrm{e}_i,$
and $V_{\Phi}^T(n):=(\Psi^T(n)\mathrm{e}_1,\ldots\Psi^T(n)\mathrm{e}_J))$. Therefore, in case {\it i)} for $k$ large enough
we have 
\begin{align*}
&\limsup_{k\to\infty} \rho^{-n_k}\sum_{i=1}^{i_k}\E\big[ \left(\Phi_{v_i}(n_k-|v_i|)\mathrm{e}_{\type(v_i)}\right)^2\ind{|\Phi_{v_i}(n_k-|v_i|)\mathrm{e}_{\type(v_i)}|>\delta\rho^{n_k/2}}\big|\cG_{{i-1}}\big]\\ 
&\le \limsup_{k\to\infty}\rho^{-n_k}\sum_{i=1}^{\infty} \Psi^T(n_k-|v_i|)\mathrm{e}_{\type(v_i)}
=\limsup_{k\to \infty}\rho^{-n_k}Z_{n_k}^{\Psi^T}\to W \sum_{k\in \Z}V^T_{\Phi}(k)\rho^{-k}\mathrm{u},
\end{align*}
almost surely as $k\to\infty$ in view of Proposition \ref{prop:appl-kesten} {\it i)} (equation  \eqref{eq: nerman's type convergence}), applied to the one dimensional deterministic characteristic $V_{\Phi}^T$. As the above holds for any $T$, taking now the limit as $T$ goes to infinity, by dominated convergence theorem we infer \eqref{eq:mgale CLT2} in case {\it i)}. The proof of  \eqref{eq:mgale CLT2} in case {\it ii)} goes exactly as in case {\it i)}, with the only difference that in the last line of the equation above  $\sum_{k\in \Z}V^T_{\Phi}(k)\rho^{-k}\mathrm{u}$ is replaced with $\lim_{n\to\infty}n^{-\gamma}\sum_{k\leq n} V_{\Phi}^T(k)\rho^{-k}\mathrm{u}$, which exists in view of the assumptions. Thus \eqref{eq:mgale CLT2} holds both in case {\it i)} and {\it ii)} and this completes the proof.	
\end{proof}

\subsection{From deterministic to random centered characteristics}

\textbf{Star construction.} Once we have proved the limit theorems for centered characteristics in Theorem \ref{thm: centred characteristic}, the  general non-centered case will be reduced to the centered case in the following way: for a deterministic characteristic $\Phi$ with a moderate growth rate of $\E Z^\Phi_n$, we apply a {\it star transformation} that we define below,  in order to obtain a centered random characteristic $\Phi^{\star}$, for which we can use Theorem \ref{thm: centred characteristic}. 

For a deterministic function $\Phi:\Z\to\C^{d\times J}$ of dimension $d$ that satisfies  $\sum_{k= 0}^{\infty}\|\Phi(-k)\|\rho^k<\infty$, we define a random characteristic  $\Phi^\star:\Z\to\C^{d\times J}$ by
setting
 \begin{align}
\label{eq:definition of phi hat-gen}
\Phi^{\star}(k)\defeq\sum_{l=0}^{\infty}\Phi(k-1-l)A^{l}(L-A),
\end{align}
and similarly, for any $u\in\U$ 
 \begin{align}
\label{eq:definition of phi hat}
\Phi^{\star}_u(k)=\sum_{l=0}^{\infty}\Phi(k-1-l)A^{l}(L(u)-A),
\end{align}
and recalling that $(L(u))_{u\in \U}$ is an i.i.d copy of the offspring distribution matrix $L$, indexed over the vertices of  $ \U$.
In the next result we gather several  properties of the star characteristic $\Phi^{\star}$.
\begin{lemma}
	\label{lem: reduction to centred}
Assume $(GW1)-(GW3)$ hold and let $\Phi:\Z\to\C^{d\times J }$ be a deterministic characteristic that satisfies  $\sum_{k=0}^{\infty}\|\Phi(-k)\|\rho^{k}<\infty$. Then for any $u\in \U$, the characteristic $\Phi^{\star}_u$ given by \eqref{eq:definition of phi hat} is   well-defined and the following hold.
\vspace{-0.25cm}
\begin{enumerate}[i)]
\setlength\itemsep{0em}
\item The random variables $((L(u),\Phi^{\star}_u))_{u\in\U}$ are independent and identically distributed.
\item For any $u\in \U$ and $k\in \Z$, we have $\E\left[\Phi^{\star}_u(k)\right]=0$.
\item If $\E[\|L\|^2]<\infty$ and  for any initial type $i_0\in [J]$ we have 
$$\sum_{k\in \Z}\big\|\E \left[Z_k^\Phi\big| Z_0={\mathrm e}_{i_0}|\right]\big\|^2\rho^{-k}<\infty$$ 
then it holds
$$\sum_{k\in \Z}\E\left[\|\Phi^{\star}(k)\|^2\right]\rho^{-k}<\infty.$$
	\item For any $n\in \N$, $Z^{\Phi^{\star}}_n$ is a re-centering of $Z^{\Phi}_n$: if for any initial type $i_0\in [J]$, it holds 
	$$\sum_{k=0}^{\infty}\big\|\E \left[Z_k^\Phi\right]\big\|\rho^{-k}
	<\infty$$
	then $Z^{\Phi^\star}_n$ is well-defined and can be written as
	$$Z^{\Phi^{\star}}_n=Z^\Phi_n-\E Z^\Phi_n.$$
\end{enumerate}
\end{lemma}
\begin{proof}
The fact that  $\Phi_u$ is $L(u)$-measurable implies \textit{i)}. By definition, \textit{ii)} is also obvious since $A=\E[L(u)]$ for every $u\in \U$.

For  \textit{iii)}   note that 
	\begin{align*}
	\big\|\sum_{l\ge 0}\Phi(k-l)A^l\big\|^2
	=\sum_{j=1}^J\left\|\E\big[Z^\Phi_k\big|Z_0=\mathrm{e}_j\big]\right\|^2.
	\end{align*}
We have
\begin{align*}
	\E\big[\|\Phi^\star(k)\|^2\big]=\E\Big[\big\|\sum_{l= 0}^{\infty}\Phi(k-1-l)A^{l}(L-A)\big\|^2\Big]
	\le\big\|\sum_{l= 0}^{\infty}\Phi(k-1-l)A^{l}\big\|^2\E\big[\|L-A\|^2\big],
	\end{align*}
	and therefore 
	\begin{align*}
	\sum_{k\in \Z}\E\left[\|\Phi^{\star}(k)\|^2\right]\rho^{-k}\le\E\big[\|L-A\|^2\big]\rho
	\sum_{k\in \Z}\rho^{-k-1}\sum_{j=1}^J\big\|\E\big[Z^\Phi_{k-1}\big|Z_0=\mathrm{e}_j\big]\big\|^2
	 <\infty,
	\end{align*}
which finally proves	\textit{iii)}. 	

In order to prove \textit{iv)}, remark first that for the random characteristic $\Psi:\Z\to\C^{d\times J}$ defined as 
$\Psi(k)\defeq \ind{k=0}L$ it holds $Z^\Psi_n=Z_{n+1}$. Thus 
	$$Z^{\Phi^\star}_n=\sum_{k= 0}^{\infty}\sum_{l= 0}^{\infty}\Phi(n-k-1-l)A^{l}(Z_{k+1}-AZ_{k})$$
	and our assumption enables us to split the above sum as
	\begin{align*}
	Z^{\Phi^\star}_n
	&=\sum_{k=1}^{\infty}\sum_{l= 0}^{\infty}\Phi(n-k-l)A^{l}Z_{k}-\sum_{k= 0}^{\infty}\sum_{l= 1}^{\infty}\Phi(n-k-l)A^{l}Z_{k}\\
	&=\sum_{k= 1}^{\infty}\Phi(n-k)Z_{k}-\sum_{l= 1}^{\infty}\Phi(n-l)A^{l}Z_0
	=Z^{\Phi}_n-\E\big[Z^{\Phi}_n\big]
	\end{align*}
and this proves the last claim.
\end{proof}

\subsection{Associated martingales}\label{sec:martingales}

The aim of this section is to investigate the  martingales $W_n^{(1)}$ and $W_n^{(2)}$ and to prove the convergence of $W_n^{(1)}$ in  Lemma~\ref{lem:convergence of martingales}.
\begin{proof}[Proof of Lemma \ref{lem:convergence of martingales}]
	It suffices to show that $\E\big[\big\|W^{(1)}_{n+1}-W^{(1)}_{n}\big\|^2\big]$ decays exponentially in $n$. For any   $\delta >0$  such that  $\sigma_A\cap\left\{z\in \C:\rho^{1/2}<|z|\le \rho^{(1+\delta)/2}\right\}=\emptyset$, there exists  a constant $C>0$ such that for any $n\in\N$, $\|A_1^{-n}\|^2\le C\rho^{-(1+\delta)n}$. We have
	
	\begin{align*}
	\E\big[\big\|W^{(1)}_{n+1}-W^{(1)}_{n}\big\|^2\big|\F_n\big]&=
	\E\Big[\big\|A_1^{-n-1}\pi^{(1)} \sum_{\mathclap{u\in \T,|u|=n}}L(u)\mathrm{e}_{\type(u)}-A_1^{-n-1}\pi^{(1)} \sum_{\mathclap{u\in \T,|u|=n}}A\mathrm{e}_{\type(u)}\big\|^2\big|\F_n\Big]\\
	&= \E\Big[\big\|A_1^{-n-1}\pi^{(1)} \sum_{|u|=n}(L(u)-A)\mathrm{e}_{\type(u)}\big\|^2\big|\F_n\Big]\\
	&=\sum_{u\in \T,|u|=n}\E\Big[\big\|A_1^{-n-1}\pi^{(1)} (L(u)-A)\mathrm{e}_{\type(u)}\big\|^2\big|\F_n\Big]\\
	&\le \sum_{u\in \T,|u|=n}C\rho^{-(1+\delta)n}\E\big[\big\|L-A\big\|^2\big]<\infty,
	\end{align*}
where the convergence of the last series follows from $A=\E[L]$ together with assumption (GW3)
and $\E\big[\big\|L-A\big\|^2\big]=\sum_{i,j\in [J]}\Var\big[L^{(i,j)}\big]$.
Thus, by taking expectations in the above conditional expectation and using the fact that the expected number of particles in the $n$-th generation is $\big<{\bf 1}, A^nZ_0\big>$, we obtain
	\begin{align*}
	\E\Big[\big\|W^{(1)}_{n+1}-W^{(1)}_{n}\big\|^2\Big]\le C'\rho^{-(1+\delta)n}\|A^nZ_0\|\le C'\rho^{-\delta n}
	\end{align*}
	which implies that 
	$$\sum_{n\geq 1}\E\Big[\big\|W^{(1)}_{n+1}-W^{(1)}_{n}\big\|^2\Big]<\infty,$$
	and this proves the claim.
\end{proof}
The next step in our approach is to express the martingale $W_n^{(1)}$ and its $\L^2$-limit  $W^{(1)}$ in terms of a branching process counted with some characteristic $\Phi^1$, that we describe below.
\begin{lemma}
	\label{lem: difference of martingale and its limit}
Assume (GW1)-(GW3) hold. Then for any $n\in\N$ we have
$$A_1^{n}\big(W^{(1)}-W^{(1)}_{n}\big)=Z^{\Phi^{1}}_n,$$
where the random characteristic $\Phi^{1}:\Z\to\C^{1\times J}$ is given by:  for $k\in \Z$
$$\Phi^{1}(k)\defeq A_1^{k-1}\pi^{(1)} (L-A)\ind{k\le 0}.$$
\end{lemma}
\begin{proof}
By telescoping, for any $n\in\N$ we obtain
\begin{align*}
	W^{(1)}-W^{(1)}_{n}&=\sum_{k= n}^{\infty}\big(W^{(1)}_{k+1}-W^{(1)}_{k}\big)
	=\sum_{k= n}^{\infty}A_1^{-k-1}\pi^{(1)} \sum_{u\in \T,|u|=k}(L(u)-A)\mathrm{e}_{\type(u)}\\
	&=\sum_{k=n}^{\infty}A_1^{-k-1}\pi^{(1)} \sum_{u\in\T}(L(u)-A)\mathrm{e}_{\type(u)}\ind{k-|u|=0}
	\end{align*}
which in turn equals
\begin{align*}	
	&=\sum_{u\in\T}A_1^{-|u|-1}\pi^{(1)} (L(u)-A)\mathrm{e}_{\type(u)}\ind{n-|u|\le 0}\\
	&=A_1^{-n}\sum_{u\in\T}A_1^{n-|u|-1}\pi^{(1)} (L(u)-A)\mathrm{e}_{\type(u)}\ind{n-|u|\le 0}\\
	&=A_1^{-n}\sum_{u\in\T}\Phi^{1}_u(n-|u|)\mathrm{e}_{\type(u)}=A_1^{-n}Z_n^{\Phi^1},
	\end{align*}
which, together with the definition of a branching process counted with a characteristic, proves the claim.	
\end{proof}

Now we switch to the analysis of $W^{(2)}$.
Let  $\mathrm{x}\in \C^{1\times J}$ be an arbitrary fixed (deterministic) row vector.  In order to understand the limit behavior of $\mathrm{x}\pi^{(2)} Z_n$ we need a  detailed analysis of the spectral decomposition of the mean offspring matrix $A$ on the critical circle $\{z\in \C:\ |z|=\sqrt\rho\}$ of radius $\sqrt{\rho}$.
If there are no eigenvalues of $A$ with absolute value equal to $\sqrt{\rho}$, then $\pi^{(2)}$ is the zero matrix, and the projection $\pi^{(2)}Z_n$ is the zero vector, so everything is trivially zero. 

\textbf{Assumption: } For the rest of this section, we assume that the matrix $A$ has at least one eigenvalue with absolute value equal to $\sqrt{\rho}$, that is, $\sigma^2_A\neq\emptyset$.

\begin{theorem}
	\label{lem:critical martingale} Assume (GW1)--(GW3) hold and for any $\mathrm{x}\in \C^{1\times J}$ and $0\le l\le J-1$, consider  $\sigma_l^2$ as in \eqref{eq:definition of sigma_l}.
If $\sigma_l=0 $ for all $0\le l\le J-1$ then $\mathrm{x}\pi^{(2)} Z_n=\mathrm{x}\pi^{(2)}A^nZ_0$. Otherwise, let $l$ be maximal such that $\sigma_l\neq 0$. Then the following stable convergence holds
	\begin{align*}
	\frac{\mathrm{x}\pi^{(2)} (Z_n-A^nZ_0)}{n^{l+\frac12}\rho^{n/2}}\stabto\sigma_{l}\sqrt W\cdot\mathcal{N}(0,1)\qquad\text{as } n\to\infty,
	\end{align*}
	where $\mathcal{N}(0,1)$ is a standard normal variable independent of $W$.
\end{theorem}
\begin{proof}
We decompose $\pi^{(2)}A=D+N$ in its semisimple part $D=\sum_{\lambda\in\sigma^2_A}\lambda\pi_\lambda$ and its nilpotent part $N=\sum_{\lambda\in\sigma^2_A}(A-\lambda I)\pi_\lambda$, and remark that $N^J=0$ and both $D$ and $N$ commute with any projection $\pi_{\lambda}$ and therefore between themselves. With the convention $N^0\defeq I$ (even for $N=0$), for $k\ge J$ we have
\begin{align*}
\pi^{(2)}A^k=(D+N)^k=\sum_{l=0}^{J-1}{k\choose l}D^{k-l}N^l.
\end{align*} 
For $\mathrm{x}\in\C^{1\times J}$, $\mathrm{x}\pi^{(2)} Z_n$ can be written as a branching process $(Z_n^{\Theta})_n$ with  characteristic $\Theta$
\begin{align*}
\mathrm{x}\pi^{(2)} Z_n=\sum_{u\in\T}\ind{n-|u|=0}\mathrm{x}\pi^{(2)}e_{\type(u)}=Z^{\Theta}_n,
\end{align*}
where $\Theta:\Z\to\C^{1\times J}$ is defined as
$$\Theta(k)\defeq\ind{k=0}\mathrm{x}\pi^{(2)}.$$ 
Since $\Theta$ is just a deterministic function  and satisfies the assumptions of Lemma~\ref{lem: reduction to centred}, we  may apply the star transformation in order to get the corresponding centered random characteristic $\Theta^{\star}:\Z\to\C^{1\times J}$:
\begin{align}
	\label{eq:Z^2 as CMJ process}
\mathrm{x}\pi^{(2)} Z_n=Z^{\Theta}_n=Z^{\Theta^\star}_n+\E Z_n^{\Theta}=Z^{\Theta^\star}_n+\mathrm{x}\pi^{(2)}A^nZ_0\ind{n\ge 0},
\end{align}
that is,  the process $\mathrm{x}\pi^{(2)}Z_n-\mathrm{x}A^n_2Z_0$ is actually a branching process counted with a centered characteristic $\Theta^{\star}$:
$$\mathrm{x}\pi^{(2)}Z_n-\mathrm{x}A^n_2Z_0=Z^{\Theta^\star}_n.$$
On the other hand, by the definition  of $\Theta^\star$, we can rewrite
\begin{equation}\label{eq:theta-star}
\Theta^\star(k)=\sum_{l=0}^{\infty}\Theta(k-1-l)A^l(L-A)=\mathrm{x}\pi^{(2)}A^{k-1}(L-A)\ind{k>0},
\end{equation}
and 
$$\Var [\Theta^\star(k)]=\big(\Var \big[\mathrm{x}\pi^{(2)}A^{k-1}L\mathrm{e}_j\ind{k>0}\big]\big)_{1\le j\le J}.$$
If for all $0\le l \le J-1$, it holds
\begin{align*}
\sum_{j=1}^J\sum_{\lambda\in\sigma^2_A}\Var\big[\mathrm{x}  \pi_\lambda N^{l}L\mathrm{e}_j\big]= 0,
\end{align*}
then
$\mathrm{x}  \pi_\lambda N^{l}L=\mathrm{x}  \pi_\lambda N^{l}A$ which in turn gives $\Theta^\star(k)=0$ and this implies $$\mathrm{x}\pi^{(2)} Z_n=\mathrm{x}\pi^{(2)}A^nZ_0\ind{n\ge 0}.$$
Therefore, we may assume that there exists at least one $l$,  $0\le l\le J-1$ such that\\ $\sum_{j=1}^J\sum_{\lambda\in\sigma^2_A}\Var\left[\mathrm{x}  \pi_\lambda N^{l}L\mathrm{e}_j\right]$ does not vanish, and we let $l^*$ to be the maximal one among those $l$'s. So we have
\begin{align*}
\sum_{j=1}^J\sum_{\lambda\in\sigma^2_A}\Var\big[\mathrm{x}  \pi_\lambda N^{l^*}L\mathrm{e}_j\big]>0
\end{align*}
and by maximality of $l^*$, for $l>l^*$, any $j\in[J]$, and any $k\geq 0$,
$$\mathrm{x}\pi_\lambda  D^{k-l}N^{l}L\mathrm{e}_j=\mathrm{x}\pi_\lambda  D^{k-l}N^{l}A\mathrm{e}_j.$$
Note that $l^*$ is at most the size of the largest Jordan block corresponding to eigenvalues $\lambda\in \sigma^2_A$. 
In order to check that the first condition in Theorem~\ref{thm: centred characteristic}~{\it ii)} is satisfied for $\Theta^{\star}$, notice that for $k\ge 0$, we have
\begin{align*}
\Var &\big[\mathrm{x}\pi^{(2)}A^{k}L\mathrm{e}_j\big]=\Var \big[\mathrm{x}\pi^{(2)}  (D+N)^{k}L\mathrm{e}_j\big]
=\Var\Big[\mathrm{x}\pi^{(2)}  \sum_{l=0}^{J-1}{k \choose l}D^{k-l}N^{l}L\mathrm{e}_j\Big]\\
&=\Var\Big[\mathrm{x} \sum_{l=0}^{J-1}{k \choose l}\sum_{\lambda\in\sigma^2_A}\rho^{(k-l)/2} e^{\imag \arg(\lambda)(k-l)}\pi_{\lambda}N^lL\mathrm{e}_j\Big]\\
&=\Var\Big[\mathrm{x} \sum_{l=0}^{l^*}{k \choose l}\sum_{\lambda\in\sigma^2_A}\rho^{(k-l)/2} e^{\imag \arg(\lambda)(k-l)}\pi_{\lambda}N^lL\mathrm{e}_j\Big]\\
& =\frac{k^{2l^*}\rho^{k-l^*}}{(l^*!)^2}\Var\Big[\mathrm{x} \sum_{\lambda\in\sigma^2_A} e^{\imag \arg(\lambda)(k-l^*)}\pi_{\lambda}N^{l^*}L\mathrm{e}_j\Big]+O\big(k^{2l^*-1}\rho^k\Big).\\
\end{align*}
Thus, for every $j\in[J]$ we have $\sup_nn^{-(2l^{\star}+1)}\sum_{k\leq n}\Var \big[\mathrm{x}\pi^{(2)}A^{k}L\mathrm{e}_j\big]\rho^{-j}<\infty$, and summing up over all $j\in[J]$ shows that the first condition  in Theorem~\ref{thm: centred characteristic}~{\it ii)} holds  for $\gamma=2l^*+1$ and the centered characteristic $\Theta^{\star}$.
We check now the second condition:
\begin{align*}
n^{-(2l^*+1)}&\sum_{k=0}^n\rho^{-k}\Var \big[\Theta^\star(k)\big]\mathrm{u}=n^{-(2l^*+1)}\sum_{j\in[J]}\sum_{ k=0}^{ n}\rho^{-k}\Var \big[\mathrm{x}\pi^{(2)}A^{k}L\mathrm{e}_j\big]\mathrm{u}_j\\
&=\sum_{j\in[J]}n^{-(2l^*+1)}\sum_{k=0}^{n}\frac{k^{2l^*}\rho^{-l^*}}{(l^*!)^2}\Var\Big[\mathrm{x} \sum_{\lambda\in\sigma^2_A} e^{\imag \arg(\lambda)(k-l^*)}\pi_{\lambda}N^{l^*}L\mathrm{e}_j\Big]\mathrm{u}_j+o(1)\\
&=\sum_{j\in[J]}\sum_{\lambda\in\sigma^2_A} \frac{\rho^{-l^*}}{(2l^*+1)(l^*!)^2}\Var \Big[\mathrm{x}  \pi_\lambda N^{l^*}L\mathrm{e}_j\Big]\mathrm{u}_j+o(1)\\
&=\frac{\rho^{-l^*}}{(2l^*+1)(l^*!)^2}\sum_{\lambda\in\sigma^2_A} \Var \left[\mathrm{x}  \pi_\lambda (A-\lambda I)^{l^*}L\right]\mathrm{u}+o(1)>0,
\end{align*}
where in the equation before the last one we have used that for $\lambda_1\neq\lambda_2$ and $j\in [J]$
\begin{align*}
\sum_{k=0}^{n}\frac{k^{2l^*}}{n^{2l^*+1}}&\Cov \Big[\mathrm{x}  e^{\imag\xi_1(k-l^*)}\pi_{\lambda_1} N^{l^*}L\mathrm{e}_j,\mathrm{x}  e^{\imag\xi_2(k-l^*)}\pi_{\lambda_2} N^{l^*}L\mathrm{e}_j\Big]\\
&=\Cov \Big[\mathrm{x} e^{-\imag\xi_1l^*}\pi_{\xi_1} N^{l^*}L\mathrm{e}_j,\mathrm{x}  e^{-\imag\xi_2l^*}\pi_{\xi_2} N^{l^*}L\mathrm{e}_j\Big]\sum_{k=0}^{ n-1}\frac{k^{2l^*}}{n^{2l^*+1}}e^{\imag k(\xi_1-\xi_2)}=o(1),
\end{align*}
with $\xi\defeq\arg(\lambda_i)$. So $\sigma_{l^{\star}}$ is given by
$$\sigma_{l^{\star}}=\frac{\rho^{-l^*}}{(2l^*+1)(l^*!)^2}\sum_{\lambda\in\sigma^2_A} \Var \Big[\mathrm{x}  \pi_\lambda (A-\lambda I)^{l^*}L\Big]\mathrm{u},$$
and with this in hands we are in the position to 
apply Theorem~\ref{thm: centred characteristic}~{\it ii)} to $Z^{\Theta^\star}_n$ and $\gamma=2l^*+1$, which in turn yields the claim
\begin{align*}
\frac{\mathrm{x}_2\pi^{(2)} (Z_n-A^nZ_0)}{n^{l^*+\frac12}\rho^{n/2}}=\frac{Z^{\Theta^\star}_n}{n^{(2l^*+1)/2}\rho^{n/2}}\stabto \sigma_{l^*}\sqrt W\cdot\mathcal{N}(0,1),\quad \text{as } n\to\infty,
\end{align*}
where by abuse of notation, in the claim of the result we still write $l$ instead of $l^{*}$, to denote the maximal $0\leq l\leq J-1$ for which $\sigma_l \neq 0$.
\end{proof}

\section{Proof of Theorem \ref{thm:main}}\label{sec:proof-main-result}
In this section we finally provide the proof of the main result. Before doing so, we sketch the main idea behind it. For a random characteristic $\Phi$, in view of the linearity of $Z_n^{\Phi}$ in   $\Phi$, we can write 
$$ Z_n^\Phi=Z_n^{\Phi-\E[\Phi]}+Z_n^{\E[\Phi]},$$
where the first term is  subject of Theorem~\ref{thm: centred characteristic} since $\Phi-\E[\Phi]$ is a centered  characteristic.  We may thus assume for the moment, without loss of generality,  that the characteristic $\Phi$ of dimension one  is deterministic.
We  decompose the branching process $Z_n^{\Phi}$ counted with deterministic characteristic $\Phi$, according to the projections $\pi^{(i)}$ for $i=1,2,3$. Since for any $k\in \Z$, $\Phi(k)=\Phi(k)(\pi^{(1)}+\pi^{(2)}+\pi^{(3)})$ we can write
\begin{align*}
Z^\Phi_n=Z^{\Phi\pi^{(1)}}_n+Z^{\Phi\pi^{(2)}}_n+Z^{\Phi\pi^{(3)}}_n.
\end{align*}
The fluctuations of $Z_n^{\Phi}-\mathrm{x_1}A^n_1 W^{(1)}-\mathrm{x_2}A^n_2Z_0$ will be then a composition of the fluctuations
\begin{enumerate}[1.]
\setlength\itemsep{0em}
	\item  between the process  $Z_n^{\Phi\pi^{(1)}}$ and the rescaled martingale $\mathrm{x_1}A^n_1 W_n^{(1)}=\mathrm{x_1}\pi^{(1)}Z_n$,
	\item between the process  $Z_n^{\Phi\pi^{(2)}}$ and the rescaled martingale $\mathrm{x_2}A^n_2 W_n^{(2)}=\mathrm{x_2}\pi^{(2)}Z_n$,
	\item between the (rescaled) martingale $\mathrm{x_1}A^n_1 W_n^{(1)}$ and its limit $\mathrm{x_1}A^n_1 W^{(1)}$,
	\item between the (rescaled) martingale $\mathrm{x_2}A^n_2 W_n^{(2)}$ and its expectation $\mathrm{x_2}A^n_2Z_0$,
	\item fluctuations of the process $Z_n^{\Phi\pi^{(3)}}$.
\end{enumerate}  
Each of these fluctuations, except the one from 4., is of magnitude $\rho^{n/2}$. The later (if nontrivial), is in view of Theorem~\ref{lem:critical martingale} of magnitude $n^{l+\frac12}\rho^{n/2}$. This is the dichotomy in Theorem~\ref{thm:main}. Now we show how each of the above processes can be expressed as branching processes counted with a centered characteristic and a deterministic error of order $o(\rho^{n/2})$. Even though the calculations are technical, they are quite straightforward.
\subsection{The first two terms: the differences  $Z^{\Phi\pi^{(i)}}_n-\mathrm{x_i}\pi^{(i)}Z_n$ for $ i=1,2$.}

This two terms can be handled together, by estimating first the growth in mean of the differences.  
Recalling the definition of the vectors $\mathrm{x}_i$ from \eqref{eq:vect-x1-x2}, for $i=1,2$, we get
	$$\E\big[\mathrm{x}_i\pi^{(i)}Z_n\big]	=\mathrm{x}_i\pi^{(i)}A^nZ_0.$$
	Therefore, for $n\in\N$ large enough we have 
	\begin{align*}
\E\big[\mathrm{x}_i\pi^{(i)}Z_n\big]	-	\E\big[ Z^{\Phi\pi^{(i)}}_n\big]&=\sum_{l\in \Z}\Phi(l)\pi^{(i)}A^{n-l}Z_0- \sum_{l=-\infty}^n\Phi(l)\pi^{(i)}A^{n-l}Z_0 \\
	&=\sum_{l=n+1}^{\infty}\Phi(l)\pi^{(i)}A^{n-l}Z_0=O\Big(\sum_{l=n+1}^{\infty}\|\Phi(l)\|\vartheta^{{n-l}}\Big)=O(\vartheta^n),
	\end{align*}
	and similarly, for negative integers $n\leq 0$
	\begin{align*}
	\E \big[Z^{\Phi\pi^{(i)}}_{n}\big]=\sum_{l=-\infty}^{n}\Phi(l)\pi^{(i)}A^{n-l}Z_0=O\Big(\sum_{l=-\infty}^{n}\|\Phi(l)\|\rho^{n-l}\Big)=O(\rho^{n}).
	\end{align*}
	Thus for $i=1,2$ it holds
	\begin{equation}\label{eq:expec-id}
	\big|\E\big[\mathrm{x}_i\pi^{(i)}Z_n\big]	-	\E\big[ Z^{\Phi\pi^{(i)}}_n\big]\big|\leq C(\rho^{-n}\wedge\vartheta^n).
	\end{equation}	
	In other words, for the characteristic $\Psi_i:\Z\to \C^{1\times J}$, $i=1,2$ defined by
	$$\Psi_i(k)\defeq\big(\Phi(k)-\ind{k=0}\mathrm{x}_i\big)\pi^{(i)},$$
	we have 
	$$Z^{\Psi_i}_n=Z^{\Phi\pi^{(i)}}_n-\mathrm{x_i}\pi^{(i)}Z_n \quad\text{and}\quad \big|\E Z^{\Psi_i}_n\big|\leq C(\rho^{n}\wedge\vartheta^n).$$
	Since
	\begin{align*}
	\sum_{k=0}^{\infty}||\Psi_i(-k)||\rho^k\leq \sum_{k=0}^{\infty}||\Phi(-k)||\rho^k+||\mathrm{x}_i\pi^{(i)}||\leq \sum_{k=0}^{\infty}||\Phi(-k)||\rho^k+\sum_{k\in \Z}||\Phi(k)||\rho^{-k}<\infty,
	\end{align*}
	the centered random characteristic $\Psi^{\star}_i:\Z\to\C^{1\times J}$ is well defined and given by
	\begin{align*}
	\Psi_i^{\star}(k)&=\sum_{l=0}^{\infty}\Psi_i(k-l-1)A^l(L-A)=\sum_{l=0}^{\infty}\left(\Phi(k-l-1)-\ind{k-1=l}\mathrm{x}_i\right)\pi^{(i)}A^l(L-A)\\
	&=\sum_{l=0}^{\infty}\Phi(k-l-1)\pi^{(i)}A^l(L-A)-\mathrm{x}_i\pi^{(i)}A^{k-1}(L-A)\ind{k>0}\\
	&=\sum_{l=0}^{\infty}\Phi(k-l-1)\pi^{(i)}A^l(L-A)-\sum_{l\in \Z}\Phi(l)\pi^{(i)}A^{k-l-1}(L-A)\ind{k>0}\\
	&=\sum_{l=0}^{\infty}\Phi(k-l-1)\pi^{(i)}A^l(L-A)-\sum_{l\in \Z}\Phi(k-l-1)\pi^{(i)}A^{l}(L-A)\ind{k>0}
	\end{align*}
	which  equals 
	\begin{equation}\label{eq:cent2}
	\Psi_i^{\star}(k)=
	\begin{cases}
	\sum_{l=0}^{\infty}\Phi(k-l-1)\pi^{(i)}A^{l}(L-A)&\quad \text{if } k\leq 0,\\
	-\sum_{l=-\infty}^{-1}\Phi(k-l-1)\pi^{(i)}A^{l}(L-A)&\quad \text{if } k> 0.
	\end{cases}
	\end{equation}
	Moreover,  by Lemma~\ref{lem: reduction to centred} 
	$$\sum_{k\in \Z}\E\left[\|\Psi_i^{\star}(k)\|^2\right]\rho^{-k}<\infty,$$
	and 
	$$Z^{\Phi\pi^{(i)}}_n-\mathrm{x_i}\pi^{(i)}Z_n=Z^{\Psi_i^\star}_n+\E Z^{\Psi_i}_n=Z^{\Psi_i^\star}_n+O(\vartheta^n),\quad \text{for }i=1,2.$$

\subsection{Martingale fluctuations $\mathrm{x}_1 A_1^{n}\big(W^{(1)}_{n}-W^{(1)}\big)$}

We consider now the fluctuations in point 3.~of the main proof sketch. For the characteristic $\Psi_3:\Z\to\C^{1\times J}$ defined as
\begin{equation}\label{eqref:psi3}
\Psi_3(k)\defeq-\mathrm{x}_1\Phi^{1}(k)= -\mathrm{x}_1A_1^{k-1}\pi^{(1)} (L-A)\ind{k\le 0},
\end{equation}
Lemma \ref{lem: difference of martingale and its limit} together with the definition of a branching process with a characteristic implies
	$$\mathrm{x}_1 A_1^{n}\left(W^{(1)}_{n}-W^{(1)}\right)=Z^{\Psi_3}_n.$$
Clearly, $\Psi_3$ is centered and
	\begin{align*}
	\sum_{k\in \Z}\E\big[\|\Psi_3(k)\|^2\big]\rho^{-k}=
	\|\mathrm{x}_1\|^2\E\left[\|L-A\|^2\right]\sum_{k=-\infty}^{0}\big\|A_1^{k-1}\pi^{(1)}\big\|^2\rho^{-k}<\infty,
	\end{align*}
	as $\|A_1^{-1}\pi^{(1)}\big\|\le\rho^{\frac12+\delta}$ for some $\delta>0$.

\subsection{The third  projection $Z^{\Phi\pi^{(3)}}_n$}

Here we investigate the fluctuations of $Z^{\Phi\pi^{(3)}}_n$, hence point 5.~in the sketch of the main proof.
Since for every $n\in\N$
$$Z^{\Phi\pi^{(3)}}_n=\sum_{k=0}^{\infty}\Phi(n-k)\pi^{(3)}Z_k=\sum_{k=-\infty}^{n}\Phi(k)\pi^{(3)}Z_{n-k},$$
 we have by the same argument that led to \eqref{eq:expec-id} together with $\|\pi^{(3)}A^{i}\|\le C\vartheta^i$, that
	\begin{align}\nonumber
	|\E Z^{\Phi\pi^{(3)}}_{n}|&=\Big|\sum_{k=-\infty}^n\Phi(k)\pi^{(3)}A^{n-k}Z_0\Big|\\
	\label{eq:aproximation of 3rd term}
	&\leq C\min\Big(\sum_{k=-\infty}^ n\|\Phi(k)\|\vartheta^{{n-k}},\sum_{k=-\infty}^ {n}\|\Phi(k)\|\rho^{{n-k}}\Big)
	\le C'( \rho^{n}\wedge\vartheta^n ).
	\end{align}	
By defining $\Psi_4:\Z\to\C^{1\times J}$ as $\Psi_4(k)\defeq\Phi(k)\pi^{(3)}$, we can write
$$Z_n^{\Psi_4}=Z_n^{\Psi_4^\star}+\E\left[Z_n^{\Psi_4}\right]=Z_n^{\Psi_4^\star}+O(\vartheta^n),$$
	where $\Psi_4^\star$ is the centered characteristic given by
\begin{equation}\label{eq:psi4}	
\Psi_4^\star(k)=\sum_{l=0}^{\infty}\Phi(k-1-l)\pi^{(3)} A^l(L-A).\end{equation}
Finally, in view of Lemma \ref{lem: reduction to centred} we have
	$$\sum_{k\in \Z}\E\left[\|\Psi_4^{\star}(k)\|^2\right]\rho^{-k}<\infty.$$

\subsubsection*{Putting all together and completing the main proof for random $\Phi$}

For the rest of this section we assume that $\Phi$ is as in Theorem~\ref{thm:main}, that is, it may be random.
Let $\Psi^\star_1,\Psi^\star_2,\Psi_3,\Psi^\star_4 $ and $\Theta^\star$ be the centered characteristics  as defined previously for the deterministic characteristic $\E\Phi$. More precisely, for $i=1,2$, $\Psi_i^{\star}$ is defined as in \eqref{eq:cent2}
with $\Phi$ being replaced by $\E\Phi$; $\Psi_3$ is already centered and defined by \eqref{eqref:psi3}; $\Psi_4^{\star}$ is defined by \eqref{eq:psi4} with $\Phi$ being replaced by $\E\Phi$; finally $\Theta^{\star}$ is defined as in \eqref{eq:theta-star}.
Then we can then decompose $Z_n^{\Phi}-\mathrm{x_1}A^n_1 W^{(1)}-\mathrm{x_2}A^n_2Z_0$, for $n\in\N$ as
\begin{align}
\nonumber
Z_n^{\Phi}-\mathrm{x_1}A^n_1 W^{(1)}-&\mathrm{x_2}A^n_2Z_0=Z_n^{\Phi}-Z_n^{\E\Phi}+Z_n^{\E\Phi}-\mathrm{x_1}A^n_1 W^{(1)}-\mathrm{x_2}A^n_2Z_0\\
\nonumber
&=Z_n^{\Phi-\E\Phi}+Z_n^{\Psi^\star_1}+Z_n^{\Psi^\star_2}+Z_n^{\Psi_3}+Z_n^{\Psi^\star_4}+
\mathrm{x}_2\pi^{(2)} (Z_n-A_2^nZ_0)
+O(\vartheta^n)\\
\label{eq:decomposition of the process}
&=Z_n^{\Phi-\E\Phi+\Psi}+\mathrm{x}_2\pi^{(2)} (Z_n-A_2^nZ_0)+O(\vartheta^n),
\end{align}
where $\Psi$ is the centered characteristic given by
$$\Psi\defeq \Psi^\star_1+\Psi^\star_2+\Psi_3+\Psi^\star_4, $$
that is,
\begin{align*}
\Psi(k)=
\sum_{l\in \Z}\E\Phi(k-l-1)A^l\mathsf{P}(k,l)(L-A),
\end{align*}where
\begin{align*}
\mathsf{P}(k,l)\defeq\begin{cases}
-\pi^{(1)}\ind{l<0}+\pi^{(2)}\ind{l\geq 0}+\pi^{(3)}\ind{l\geq 0},\quad \text{if }k\leq 0\\
-\pi^{(1)}\ind{l<0}-\pi^{(2)}\ind{l<0}+\pi^{(3)}\ind{l\geq 0},\quad \text{if }k>0
\end{cases}
\end{align*}
Note that $$\sum_{k\in \Z}\E\left[\|\Psi(k)\|^2\right]\rho^{-k}<\infty,$$
as it is true for each of the four terms $\Psi^\star_1$, $\Psi^\star_2$, $\Psi_3$ and $\Psi^\star_4$ appearing in  $\Psi$,
and $\Psi$ satisfies the assumptions of Theorem \ref{thm: centred characteristic}. Therefore  
\begin{align}
\label{eq:definition of sigma}
\sigma^2\defeq\sum_{k\in \Z}\Var\left[\Phi(k)+\Psi(k)\right]\rho^{-k}\mathrm{u},
\end{align} 
is well defined. In particular, when $\Phi$ is independent of the offspring distribution matrix $L$ we have
\begin{align*}
\sigma^2=\sum_{k\in \Z}\Var\left[\Phi(k)\right]\rho^{-k}\mathrm{u}+\sum_{k\in \Z}\Var\left[\Psi(k)\right]\rho^{-k}\mathrm{u}.
\end{align*} 
For $k\in \Z$, if we define matrices $\mathsf{B}(k)$ by: 
\begin{align*}
\mathsf{B}(k)\defeq\sum_{l\in \Z}\E\left[\Phi(k-l-1)\right]A^l\mathsf{P}(k,l),
\end{align*}
then 
$\Var\left[\Psi(k)\right]
=\Big(\mathsf{B}(k)\Cov\big[L^{(1)}\big]\mathsf{B}(k)^\top,\dots,\mathsf{B}(k)\Cov\big[L^{(J)}\big]\mathsf{B}(k)^\top\Big).
$

Now  we can finally complete the proof of our main result.
\begin{proof}[Proof of Theorem~\ref{thm:main}]
	We use the decomposition \eqref{eq:decomposition of the process} of $Z_n^{\Phi}$. For the process $Z_n^{\Phi-\E\Phi+\Psi}$ we apply Theorem~\ref{thm: centred characteristic} \textit{i)} and  conclude that 
	\begin{align*}
	\frac{Z_n^{\Phi-\E\Phi+\Psi}}{\rho^{n/2}}\stabto \sigma \sqrt W\cdot\mathcal{N}(0,1), \quad \text{as }n\to\infty,
	\end{align*}
	with $\sigma$ defined by \eqref{eq:definition of sigma}.
	If  $\sigma_l=0$ for all $0\leq l\leq J-1$, then by Theorem~\ref{lem:critical martingale} we obtain that
	\begin{align*}
	\frac{Z_n^{\Phi}-\mathrm{x_1}A^n_1 W^{(1)}-\mathrm{x_2}A^n_2Z_0}{\rho^{n/2}}=\frac{Z_n^{\Phi-\E\Phi+\Psi}+O(\vartheta^n)}{\rho^{n/2}}\stabto  \sigma \sqrt W\cdot\mathcal{N}(0,1),\quad \text{as } n\to\infty
	\end{align*}
	by using Slutsky's theorem. Moreover, if also $\sigma=0$ then  $Z_n^{\Phi}-\mathrm{x_1}A^n_1 W^{(1)}-\mathrm{x_2}A^n_2Z_0$ is deterministic of growth $O(\vartheta^n)$.
	 
	 Finally, let $l=\max\{1\leq j\leq J-1:\  \sigma_j>0\}$. Then again by Slutsky's theorem together with Theorem~\ref{lem:critical martingale}, we obtain
	 \begin{align*}
	 \frac{Z_n^{\Phi}-\mathrm{x_1}A^n_1 W^{(1)}-\mathrm{x_2}A^n_2Z_0}{n^{l+\frac 12}\rho^{n/2}}=\frac{Z_n^{\Phi-\E\Phi+\Psi}+\mathrm{x}_2\pi^{(2)} (Z_n-A_2^nZ_0)+O(\vartheta^n)}{n^{l+\frac 12}\rho^{n/2}}\stabto\sigma_l  \sqrt W\cdot\mathcal{N}(0,1),
	 \end{align*}
where in both cases $\mathcal{N}(0,1)$ is a standard normal random variable independent of $W$, and this completes  the proof.
\end{proof}
\begin{remark}
	\label{rem:almost sure behaviour}
Let us emphasize that the decomposition in \eqref{eq:decomposition of the process} together with \eqref{eq:Z^2 as CMJ process} allows us to write $Z_n^{\Phi}-\mathrm{x_1}A^n_1 W^{(1)}-\mathrm{x_2}A^n_2Z_0$, up to an error or order $O(\vartheta^n)$,  as a  CMJ process $Z^{\Phi'}_n$ for some centered characteristic $\Phi'$ such that $k\mapsto \|\Var[\Phi'(k)]\|\rho^{-k}$ is summable or $\sum^N_{k=0}\|\Var[\Phi'(k)]\|\rho^{-k}=O( N^{2l+1})$. This, in turn, gives that $\E\big[(Z^{\Phi'}_n)^2\big]=\E\big[Z^{\Var\Phi'}_n\big]$ is either of order $O(\rho^n)$ or $O(n^{2l+1}\rho^n)$. In both cases, for any $\theta>\sqrt\rho$  Chebyshev's inequality together with Borel–Cantelli lemma yields $Z^{\Phi'}_n=o(\theta^n)$ almost surely, that is
	$$Z_n^{\Phi}-\mathrm{x_1}A^n_1 W^{(1)}-\mathrm{x_2}A^n_2Z_0=o(\theta^n)\quad\text{almost surely.}$$
We leave the details of this calculations to the interested reader.	
\end{remark}

\subsection{Application}

In \cite{kesten-stigum-add}, the authors investigated additional limit theorems for $\langle Z_n,\mathrm{a}\rangle$, where $(Z_n)_{n\in \N}$ is a supercritical and positively regular multitype Galton-Watson process with $J$ number of types, such that all $Z_n^i$, $1\leq i\leq J$ have finite second moments, and $\mathrm{a}\neq 0$ is a vector in $\R^J$ such that $\langle\mathrm{a},\mathrm{u}\rangle =0$. 
In this case, Kesten-Stigum theorem \cite{kesten1966} ensures that $\rho^{-n}\langle Z_n,\mathrm{a}\rangle$ converges to 0 almost surely. The authors have then provided in \cite{kesten-stigum-add} a correct normalization and studied the limit.
More precisely, they have identified all the deterministic leading terms in the expansion of  $\langle Z_n,\mathrm{a}\rangle$.
We show below, by applying our main Theorem \ref{thm:main}, how to find all the terms (not just the  leading ones) in the asymptotic expansion of  $\langle Z_n,\mathrm{a}\rangle$ up to Gaussian fluctuations, and this, of course, extends \cite{kesten-stigum-add}.

We consider  the deterministic characteristic $\Phi:\Z\to\C^{1\times J}$ defined as:
\begin{equation}\label{eq:char-kesten}
\Phi(k)=\mathrm{a}\ind{k=0},
\end{equation}
for row vector $\mathrm{a}\in \R^{1\times J}$. For such a characteristic both conditions \eqref{eq:main-assum1} and \eqref{eq:main-assum2} are trivially satisfied and the vectors $\mathrm{x}_1$ and $\mathrm{x}_2$ from \eqref{eq:vect-x1-x2} become then
$$\mathrm{x}_1=\mathrm{a}\pi^{(1)}\quad \text{and} \quad \mathrm{x}_2=\mathrm{a}\pi^{(2)}.$$
By applying Theorem \ref{thm:main} we obtain the following.
\begin{corollary}\label{cor:kest-stig}
Let $\mathrm{a}\in \R^{1\times J}$ and suppose  $(GW1)-(GW3)$ hold. Consider the characteristic $\Phi$ from \eqref{eq:char-kesten} and its associated branching process $(Z_n^{\Phi})$ counted with this characteristic. Then the following stable convergence holds.
\vspace{-0.25cm}
	\begin{enumerate}[i)]
	\setlength\itemsep{0em}
		\item If $\sigma_l=0$ for all $0\le l\leq J-1$,  with $\sigma_l$ given as in \eqref{eq:definition of sigma_l}, then there exists a constant $\sigma \ge0$ such that either $\sigma >0$ and
		\begin{align*}
		\frac{\mathrm{a}\left(Z_n-\pi^{(1)}A_1^n W^{(1)}-\pi^{(2)}A_2^nZ_0 \right)}{\rho^{n/2}}\stabto \sigma  \sqrt W\cdot\mathcal{N}(0,1),\quad \text{as } n\to\infty
		\end{align*} 
		or $\sigma =0$ and the left hand side above is a deterministic sequence converging to 0.
		\item Otherwise, let $0\le l\le J-1$ be maximal such that $\sigma_l\neq0$. Then
		\begin{align*}
		\frac{\mathrm{a}\left(Z_n-\pi^{(1)}A_1^n W^{(1)}-\pi^{(2)}A_2^nZ_0 \right)} {n^{l+\frac 12}\rho^{n/2}}\stabto \sigma_l \sqrt W\cdot\mathcal{N}(0,1), \quad \text{as } n\to\infty.
		\end{align*}
		In both cases $\mathcal{N}(0,1)$ is a standard normal random variable independent of $W$.
	\end{enumerate}
\end{corollary}

\begin{remark}
	For any $\mathrm{a}\in \mathbb{R}^J$, the variances  $\sigma^2,\sigma^2_l$ are given by $\sigma^2=\mathrm{a}\Sigma\mathrm{a}^{\mathsf T}$ and $\sigma^2_l=\mathrm{a}\Sigma_l\mathrm{a}^{\mathsf T}$ for some covariance matrices $\Sigma$ and $\Sigma_l$ that depend only on the matrices $\Cov\big[L^{(j)}\big]$ for $j\in J$ and the matrix $A$. Indeed, we have
	\begin{align*}
		\sigma^2&=\sum_{k\in \Z}\Var\big[\mathrm{a}A^{k-1}\mathsf{P}(k,k-1)(L-A)\big]\rho^{-k}\mathrm{u}\\
		&=\mathrm{a}\sum_{k\in \Z}\mathsf{B}(k)\Big(\sum_{j=1}^J\mathrm{u}_j\Cov\big[L^{(j)}\big]\Big)\mathsf{B}(k)^{\mathsf{T}}\rho^{-k}\mathrm a^{\mathsf T},
	\end{align*} 
	where, for any $k\in \Z$, $\mathsf{B}(k)=\mathrm{a}A^{k-1}\mathsf{P}(k,k-1)$ and $\mathsf{P}(k,k-1)$ is given by
	\begin{align*}
		\mathsf{P}(k,k-1)=
		\begin{cases}
			-\pi^{(1)},&\quad \text{if }k\leq 0\\
			\pi^{(3)},&\quad \text{if }k>0.
		\end{cases}
	\end{align*}
	Therefore, in a simplified way, we can write
	$$\mathsf{B}(k)=A^{k-1}\big(-\pi^{(1)}\ind{k\leq 0}+\pi^{(3)}\ind{k>0}\big)
	$$
	and $\sigma^{2}$ reduces then to $\sigma^2=\mathrm a^{\mathsf T}\Sigma\mathrm a $ with
	\begin{align*}
		\Sigma&=\sum_{k=1}^{\infty}\rho^{-k}A^{k-1}\pi^{(3)}\Big(\sum_{j=1}^J\mathrm{u}_j\Cov\big[L^{(j)}\big]\Big)(A^{k-1}\pi^{(3)})^{\mathsf{T}}\\
		&\quad+\sum_{k\leq 0}\rho^{-k}A^{k-1}\pi^{(1)}\Big(\sum_{j=1}^J\mathrm{u}_j\Cov\big[L^{(j)}\big]\Big)(A^{k-1}\pi^{(1)})^{\mathsf{T}}.
	\end{align*}
	For $\sigma_l$ we have
	\begin{align*}
		\sigma_l^2\defeq \frac{\rho^{-l}}{(2l+1)(l!)^2}\sum_{\lambda\in\sigma^2_A}\Var \big[\mathrm{a} \pi_\lambda N_\lambda^{l}L\big]\mathrm{u}
	\end{align*}
that is $$\Sigma_l=\frac{\rho^{-l}}{(2l+1)(l!)^2}\sum_{\lambda\in\sigma^2_A}\sum_{i\in[J]}\pi_\lambda N_\lambda^{l}\Cov \big[L^i\big]\mathrm{u_i}\big(\pi_\lambda N_\lambda^{l}\big)^{\mathsf{T}}.$$
\end{remark}
\begin{remark}
We now explain how to use Corollary \ref{cor:kest-stig} in order to recover the results from \cite{kesten-stigum-add}. First observe that 
	\begin{align*}
		\mathrm{a}\pi^{(1)}A_1^n W^{(1)}=\sum_{\lambda\in\sigma_A^1}\mathrm{a}\pi_\lambda A^n W^{(1)}
		=\rho^n\mathrm{a}\mathrm{u}W+\sum_{\lambda\in\sigma_A^1\setminus\{\rho\}}\mathrm{a}\pi_\lambda A^n W^{(1)},
	\end{align*}
and each of the summation terms above, for $\lambda\in\sigma_A^1\setminus\{\rho\}$, can be further decomposed into
$$\mathrm{a}\pi_\lambda A^n W^{(1)}=\mathrm{a}(\lambda I+N_\lambda)^n \pi_\lambda W^{(1)}
=\sum_{k=0}^{J-1} {n\choose k}\lambda^{n-k}\mathrm{a}N_{\lambda}^k\pi_\lambda W^{(1)}.$$
This shows that $\mathrm{a}\pi^{(1)}A_1^n W^{(1)}$ can be written as linear combination of terms $n^k\lambda^n$ with possibly random coefficients. Moreover, if the coefficient of $n^k\lambda^n$ is not deterministic then $\mathrm{a}N_{\lambda}^\gamma\pi_\lambda L$ is not deterministic for some $\gamma\ge k$.
 
Let $(\theta,\gamma)\in[\sqrt\rho,\rho]\times \N_0$ be the largest pair (in lexicographical order) such that there is $\lambda\in\sigma_A$ with $|\lambda|=\rho$ and   
$$\mathrm{a}N_{\lambda}^\gamma\pi_\lambda A\neq\mathrm{a}N_{\lambda}^\gamma\pi_\lambda L$$ with positive probability, provided such a pair exists. If $\theta>\sqrt\rho$ then from Remark~\ref{rem:almost sure behaviour} we get 
$$\mathrm{a}Z_n=n^\gamma\theta^n\sum_{\lambda\in\sigma_A,|\lambda|=\theta}\Big(\frac{\lambda}{\theta}\Big)^n
\mathrm{a}N_{\lambda}^{\gamma}\pi_\lambda W^{(1)}+o_{a.s.}(n^\gamma\theta^n),$$
that is 
\begin{align}
	\frac{\mathrm{a}Z_n-\E[\mathrm{a}Z_n]}{n^\gamma\theta^n}-\sum_{\lambda\in\sigma_A,|\lambda|=\theta}\Big(\frac{\lambda}{\theta}\Big)^n
	\mathrm{a}N_{\lambda}^{\gamma}\pi_\lambda\big( W^{(1)}-Z_0\big)\AS 0.
\end{align}

If $\theta=\sqrt\rho$, then $\gamma=l$ for $l$ as in Corollary \ref{cor:kest-stig} \textit{ii)} (with $\sigma_l$ as above). In this case 
$ \mathrm {a} A^k\pi^{(1)}(L-A)=0$ almost surely for any $k\in\Z$.
In particular,
\begin{align*}
	\mathrm{a}A^k\pi^{(1)}Z_{n+1}=\sum_{u\in \T; |u|=n}\mathrm{a}A^kL(u)\mathrm{e}_{\type(u)}=\sum_{u\in \T; |u|=n}\mathrm{a}A^kA\mathrm{e}_{\type(u)}=\mathrm{a}A^{k+1}\pi^{(1)}Z_{n}
\end{align*}
and so recursively
\begin{align*}
	\mathrm{a}\pi^{(1)}Z_{n}=\mathrm{a}A^{n}\pi^{(1)}Z_{0}.
\end{align*} 
On the other hand, according to Lemma 	\ref{lem: difference of martingale and its limit}
$$\mathrm{a} A_1^{n}\big(W^{(1)}-W^{(1)}_{n}\big)=Z^{\mathrm{a}\Phi^{1}}_n=0,$$
since $\mathrm{a}\Phi^{1}(k)=\mathrm{a}A_1^{k-1}\pi^{(1)} (L-A)\ind{k\le 0}=0$. Putting these two facts together we conclude
\begin{align*}
\mathrm{a}\pi^{(1)} A^{n}W^{(1)}=	\mathrm{a} \pi^{(1)}A^{n}W^{(1)}_n=\mathrm{a} \pi^{(1)}Z_n=\mathrm{a}A^{n}\pi^{(1)}Z_{0}.
\end{align*}
Hence
\begin{align*}
	\E [aZ_n]&=\mathrm{a}A^{n}\pi^{(1)}Z_{0}+\mathrm{a}A^{n}\pi^{(2)}Z_{0}+\mathrm{a}A^{n}\pi^{(3)}Z_{0}
	\\&=\mathrm{a}A^{n}\pi^{(1)}W^{(1)}+\mathrm{a}A^{n}\pi^{(2)}Z_{0}+o(\rho^{n/2})
\end{align*}
and from Corollary \ref{cor:kest-stig} {\it ii)} we infer
\begin{align*}
	\frac{\mathrm{a}Z_n-\E[\mathrm{a}Z_n]} {n^{\gamma+\frac 12}\rho^{n/2}}\distto \sigma_\gamma \sqrt W\cdot\mathcal{N}(0,1).
\end{align*}

If there is no such pair $(\theta,\gamma)$, then as before $\E [aZ_n]=\mathrm{a}A^{n}\pi^{(1)}Z_0+\mathrm{a}A^{n}\pi^{(2)}Z_{0}+o(\rho^{n/2})$, and  Corollary \ref{cor:kest-stig} {\it i)} implies
\begin{align*}
	\frac{\mathrm{a}Z_n-\E[\mathrm{a}Z_n]} {\rho^{n/2}}\distto \sigma\sqrt W\cdot\mathcal{N}(0,1).
\end{align*}
Finally, if all the variances vanish, then again by Corollary \ref{cor:kest-stig}, we get  $\mathrm{a}Z_n=\E[\mathrm{a}Z_n]$ and this reproves  \cite[Theorem 2.4]{kesten-stigum-add}.
\end{remark}

%

\paragraph{Acknowledgments.}We would like to thank the anonymous referee for the  careful reading and for the valuable suggestions that led to a substantial improvement of the paper.
The research of Ecaterina Sava-Huss is supported by the Austrian Science Fund (FWF): P 34129.

\bibliography{GW}{}
\bibliographystyle{abbrv}

\textsc{Konrad Kolesko}, Department of Mathematics, University of Gießen, Germany.\\
\texttt{Konrad.Kolesko@math.uni-giessen.de}

\textsc{Ecaterina Sava-Huss}, Institut für Mathematik, Universität Innsbruck, Austria.\\
\texttt{Ecaterina.Sava-Huss@uibk.ac.at}

\end{document}